\documentclass[11pt]{amsart} 
\usepackage{fullpage, amsmath, amssymb, mathrsfs, url}
\usepackage[all]{xypic}

\theoremstyle{definition}
\newtheorem{theorem}[equation]{Theorem}
\newtheorem{lemma}[equation]{Lemma}
\newtheorem{proposition}[equation]{Proposition}
\newtheorem{corollary}[equation]{Corollary}

\newtheorem{definition}[equation]{Definition}
\newtheorem{notation}[equation]{Notation}
\newtheorem{example}[equation]{Example}

\newtheorem{conjecture}[equation]{Conjecture}
\newtheorem{remark}[equation]{Remark}

\numberwithin{equation}{section}

\newcommand{\tpl}{\operatorname{\mathsf{top}}}
\newcommand{\DmodX}{\caD-\text{mod}_X}
\newcommand{\DbmodX}{D^b(\caD-\text{mod}_X)}

\newcommand{\DbmodY}{D^b(\caD-\text{mod}_Y)}
\newcommand{\Eu}{\operatorname{Eu}}
\newcommand{\Vect}{\operatorname{Vect}}

\newcommand{\Id}{\operatorname{Id}}
\newcommand{\reg}{{\operatorname{reg}}}
\newcommand{\red}{{\operatorname{red}}}
\newcommand{\img}{\operatorname{im}}
\newcommand{\rk}{\operatorname{rk}}
\newcommand{\ad}{\operatorname{ad}}
\newcommand{\HH}{\mathsf{HH}}
\newcommand{\Der}{\operatorname{Der}}
\newcommand{\Hilb}{\operatorname{Hilb}}

\newcommand{\pt}{\text{pt}}
\newcommand{\HP}{\operatorname{\mathsf{HP}}}
\newcommand{\SiSu}{\operatorname{\mathsf{Ch}}}
\newcommand{\IC}{\operatorname{IC}}
\newcommand{\IH}{\operatorname{IH}}
\newcommand{\bk}{\mathbf{k}}
\newcommand{\bC}{\mathbf{C}}

\newcommand{\bZ}{\mathbf{Z}}
\newcommand{\bF}{\mathbf{F}}

\newcommand{\bP}{\mathbf{P}}
\newcommand{\iso}{\mathop{\to}^\sim}
\newcommand{\caB}{\mathcal{B}}

\newcommand{\caO}{\mathcal{O}}

\newcommand{\caD}{\mathcal{D}}

\newcommand{\mfh}{\mathfrak{h}}
\newcommand{\mfb}{\mathfrak{b}}
\newcommand{\tor}{{\mathrm{tor}}}
\newcommand{\Hom}{\operatorname{Hom}}
\newcommand{\End}{\operatorname{End}}

\newcommand{\Ext}{\operatorname{Ext}}

\newcommand{\mfg}{\mathfrak{g}}

\newcommand{\SL}{\mathsf{SL}}
\newcommand{\GL}{\mathsf{GL}}
\newcommand{\Sp}{\mathsf{Sp}}
\newcommand{\SpAut}{\mathsf{SpAut}}
\newcommand{\mfsp}{\mathfrak{sp}}
\newcommand{\gr}{\operatorname{\mathsf{gr}}}
\newcommand{\Lie}{\operatorname{\mathsf{Lie}}}
\newcommand{\Spec}{\operatorname{\mathsf{Spec}}}

\newcommand{\onto}{\twoheadrightarrow}
\newcommand{\into}{\hookrightarrow}
\newcommand{\Sym}{\operatorname{\mathsf{Sym}}}
\newcommand{\bA}{\mathbf{A}}

\def\MR#1{}
\begin{document}

\subjclass[2010]{14F10, 37J05, 32S20, 53D55} \keywords{Hamiltonian
  flow, complete intersections, Milnor number, D-modules, Poisson
  homology, Poisson varieties, Poisson homology, Poisson traces,
  Milnor fibration, Calabi--Yau varieties, deformation quantization,
  Kostka polynomials, symplectic resolutions, twistor deformations}

\title{Poisson traces, D-modules, and symplectic resolutions}
\author{Pavel Etingof and Travis Schedler}
\begin{abstract}
  We survey the theory of Poisson traces (or zeroth Poisson homology)
  developed by the authors in a series of recent papers.  The goal is
  to understand this subtle invariant of (singular) Poisson varieties,
  conditions for it to be finite-dimensional, its relationship to the
  geometry and topology of symplectic resolutions, and its
  applications to quantizations. The main technique is the study of a
  canonical D-module on the variety. In the case the variety has
  finitely many symplectic leaves (such as for symplectic
  singularities and Hamiltonian reductions of symplectic vector spaces
  by reductive groups), the D-module is holonomic, and hence the space of
  Poisson traces is finite-dimensional.  As an application, there are
  finitely many irreducible finite-dimensional representations of
  every quantization of the variety.  Conjecturally, the D-module is
  the pushforward of the canonical D-module under every symplectic
  resolution of singularities, which implies that the space of Poisson
  traces is dual to the top cohomology of the resolution.  We explain
  many examples where the conjecture is proved, such as symmetric
  powers of du Val singularities and symplectic surfaces and Slodowy
  slices in the nilpotent cone of a semisimple Lie algebra.  We
  compute the D-module in the case of surfaces with isolated
  singularities, and show it is not always semisimple.  We also
  explain generalizations to arbitrary Lie algebras of vector fields,
  connections to the Bernstein--Sato polynomial, relations to
  two-variable special polynomials such as Kostka polynomials and
  Tutte polynomials, and a conjectural relationship with deformations
  of symplectic resolutions. In the appendix we give a brief
  recollection of the theory of D-modules on singular varieties that
  we require.
\end{abstract}

\maketitle

\section{Introduction}
\subsection{} This paper gives an introduction to the theory of traces
on Poisson algebras developed by the authors in a series of recent
papers
(\cite{ESsym,ESdm,hp0weyl,ES-dmlv,ES-ciis,PS-pdrhhvnc,BS-KpncSfc,BS-bfun}). It
is based on two minicourses given by the authors at Cargese (2014) and
ETH Zurich (2016).

Let $A$ be a Poisson algebra over $\bC$, for example,
$A=\mathcal O(X)$, where $X$ is an affine Poisson variety.  A
\emph{Poisson trace} on $A$ is a linear functional $A \rightarrow \bC$
which annihilates $\{A, A\}$, i.e., a Lie algebra character of $A$.
The space of such traces is the dual, $\HP_0(A)^*$, to the
\emph{zeroth Poisson homology}, $\HP_0(A) := A /\{A, A\}$, the
abelianization of $A$ as a Lie algebra (where $\{A,A\}$ denotes the
$\bC$-linear span of Poisson brackets of functions).

The space $\HP_0(A)$ is an important but subtle invariant of $A$. For
example, it is a nontrivial question when $\HP_0(A)$ is finite
dimensional. Indeed, even in the simple case $A=\caO(V)^G$, where $V$
is a symplectic vector space and $G$ a finite group of symplectic
transformations of $V$, finite dimensionality of $\HP_0(A)$ used to be
a conjecture due to Alev and Farkas \cite{AF-FgaPa}. It is even harder
to find or estimate the dimension of $\HP_0(A)$; this is, in general,
unknown even for $A=\caO(V)^G$.

The first main result of the paper is a wide generalization of the
Alev-Farkas conjecture, stating that $\HP_0(A)$ is finite dimensional
if $X:=\Spec A$ is a Poisson variety (or, more generally, scheme of
finite type) with finitely many symplectic leaves. Namely, the
Alev-Farkas conjecture is obtained in the special case $X=V/G$. A more
general example is $X=Y/G$ where $Y$ is an affine symplectic variety,
and $G$ is a finite group of automorphisms of $Y$ (such as symmetric
powers of affine symplectic varieties). But there are many other
examples, such as Hamiltonian reductions of symplectic vector spaces
by reductive groups acting linearly and affine symplectic
singularities (which includes nilpotent cones and Slodowy slices,
hypertoric varieties, etc.). This result can be applied to show that
any quantization of such a variety has finitely many irreducible
finite dimensional representations (at most $\dim \HP_0(\caO(X))$).

The proof of this result is based on the theory of D-modules (as
founded in 1970--1 by Bernstein and Kashiwara in, e.g.,
\cite{Ber-mordo,Kas-asspde}). Namely, we define a canonical D-module
on $X$, denoted $M(X)$, such that $HP_0(\caO(X))$ is the underived
direct image $H^0 \pi_{*}M(X)$ under the map $\pi: X\to \mathrm{pt}$.
Namely, if $i: X\hookrightarrow V$ is a closed embedding of $X$ into
an affine space, then $M(X)$, regarded as a right $\caD(V)$-module
supported on $i(X)$, is the quotient of $\mathcal{D}(V)$ by the right
ideal generated by the equations of $i(X)$ in $V$ and by Hamiltonian
vector fields on $X$.  We show that, if $X$ has finitely many
symplectic leaves, then $M(X)$ is a holonomic D-module (which extends
well-known results on group actions on varieties). Then a standard
result in the theory of D-modules implies that
$\HP_0(\mathcal O(X))=H^0\pi_{*}M(X)$ is finite dimensional.

In fact, this method can be applied to a more general problem, when we
have an affine variety $X$ acted upon by a Lie algebra $\mfg$. We say
that $X$ has finitely many $\mfg$-leaves if $X$ admits a finite
stratification with locally closed connected strata $X_i$ (called
$\mfg$-leaves) which carry a transitive action of $\mfg$ (i.e. $\mfg$
surjects to each tangent space of $X_i$). Then we show that if $X$ has
finitely many $\mfg$-leaves then the space of coinvariants
${\mathcal O}(X)/\lbrace{ \mfg,\mathcal O(X)\rbrace}$ is finite
dimensional. The previous result on Poisson varieties is then
recovered if $\mfg$ is ${\mathcal O}(X)$. The proof of this more
general result is similar: we define a canonical D-module $M(X,\mfg)$,
obtained by dividing $\mathcal{D}(V)$ by the equations of $X$ and by
$\mfg$, and show that its underived direct image
$H^0 \pi_{*}M(X,\mfg)$ to a point is
${\mathcal O}(X)/\lbrace{ \mfg,\mathcal O(X)\rbrace}$ and that it is
holonomic if $X$ has finitely many $\mfg$-leaves. In this setting, the
result is well-known if the action of $\mfg$ integrates to an action
of a connected algebraic group $G$ with $\Lie G = \mfg$, and in this
case $M(X,\mfg)$ is in fact regular (see, e.g., \cite[Lemma
1]{Tan-holsys}, \cite[Theorem 4.1.1]{Kas-Dmrt}, and \cite[Section
5]{Hot-edm}); cf.~Remark \ref{r:conj-reg-ss} below.

Moreover, the definition of $M(X,\mfg)$ makes sense when $X$ is not
necessarily affine, and $\mfg$ is a presheaf of Lie algebras on $X$
which satisfies a $\mathcal D$-localizability condition:
$\mfg(U)\mathcal D(U')=\mfg(U')\mathcal D(U')$ for any open affines
$U'\subset U\subset X$. This condition is satisfied, in particular,
when $X$ is Poisson and $\mfg={\mathcal O}(X)$.  Furthermore, it is
interesting to consider the full direct image $\pi_*M(X,\mfg)$. Its
cohomology $H^i(\pi_*M(X,\mfg))$ then ranges between $i=-\dim X$ and
$i=\dim X$, and we call it the $\mfg$-de Rham homology of $X$.  If
$\mfg=\mathcal O(X)$ for a Poisson variety $X$, we call this
cohomology the Poisson-de Rham homology of $X$, denoted
$\HP_{-i}^{DR}(X)$. For instance, if $X$ is affine, then
$\HP_0^{DR}(X) \cong \HP_0(\mathcal O(X))$.

The rest of the paper is dedicated to the study of the D-module $M(X)$
and the Poisson-de Rham homology (in particular,
$\HP_0(\mathcal O(X))$ when $X$ is affine) for specific examples of
Poisson varieties $X$. One of the main cases of interest is the case
when $X$ admits a symplectic resolution $\rho: \widetilde X\to X$. In
this case we conjecture that $M(X)\cong\rho_*\Omega_{\widetilde
  X}$.
Namely, since $\rho$ is known to be semismall, $\rho_*\Omega_X$ is a
semisimple regular holonomic D-module (concentrated in the
cohomological degree $0$), and one can show that it is isomorphic to
the semisimplification of a quotient of $M(X)$,
so the conjecture is that this quotient
is by zero and that $M(X)$ is semisimple. This conjecture
implies that
$\dim \HP_0(\mathcal O(X))=\dim H^{\dim X}(\widetilde X,\bC)$ for
affine $X$, and, more generally,
$\HP_i^{DR}(X)\cong H^{\dim X-i}(\widetilde X,\bC)$.

We discuss a number of cases when this conjecture is known: symmetric
powers of symplectic surfaces with Kleinian singularities, Slodowy
slices of the nilpotent cone of a semisimple Lie algebra, and hypertoric
varieties. However, the conjecture is open for an important class of
varieties admitting a symplectic resolution: Nakajima quiver
varieties.

It turns out that an explicit calculation of $M(X)$ and the Poisson-de
Rham homology of $X$ (in particular, $\HP_0(\caO(X))$ for affine $X$) is
sometimes possible even when $X$ does not admit a symplectic
resolution. Namely, one can compute these invariants for symmetric
powers of symplectic varieties of any dimension, and for
arbitrary complete intersection surfaces in $\bC^n$ with isolated
singularities. We discuss these calculations at the end of the paper.
For example, it is interesting to compute $M(X)$ when $X$ is the cone
of a smooth curve $C$ of degree $d$ in $\bP^2$.  Recall that the genus
of $C$ is $(d-1)(d-2)/2$, and the Milnor number of $X$ is
$\mu=(d-1)^3$. We show that $M(X) \cong \delta^{\mu-g}\oplus M(X)_{\text{ind}}$, where
$M(X)_{\text{ind}}$ is an indecomposable D-module containing $\delta^{2g}$, such
that $M(X)_{\text{ind}}/\delta^{2g}$ is an indecomposable extension of $\delta^g$
by the intersection cohomology D-module $\IC(X)$. (Here $\delta$ is
the $\delta$-function D-module supported at the vertex of $X$).

The structure of the paper is as follows. In Section 2 we define
$\mfg$-leaves of a variety with an $\mfg$-action, state the finite
dimensionality theorem for coinvariants on varieties with finitely
many leaves, and give a number of examples (notably in the Poisson
case). In Section 3 we apply this theorem to proving that
quantizations of Poisson varieties with finitely many leaves have
finitely many irreducible finite dimensional representations. In
Section 4 we prove the finite dimensionality theorem using
D-modules. In Section 5 we define the $\mfg$-de Rham and Poisson-de
Rham homology. In Section 6 we discuss the conjecture on the
Poisson-de Rham homology of Poisson varieties admitting a symplectic
resolution. In Section 7 we discuss Poisson-de Rham homology of
symmetric powers. In Section 8 we discuss the structure of $M(X)$ when
$X$ is a complete intersection with isolated singularities. In Section
9 we discuss weights on the Poisson-de Rham homology (and hence
$\HP_0$) of cones, and state an enhancement of the aforementioned
conjecture in this case which incorporates weights.  Finally, in the
appendix we review background on D-modules used in the body of the
paper.

{\bf Acknowledgements.} The work of P.~E.~was supported by the NSF
grant DMS-1502244. The work of T.~S.~was supported by the NSF grant
DMS-1406553. We thank Y.~Namikawa for pointing out a result in Slodowy's notes \cite{Slflsgs} (see Section \ref{ss:conj-desc}). We are grateful to T.~Bitoun for pointing out some errors in Section \ref{ss:RBSp}. Thanks also to the anonymous referees for carefully reading and providing useful comments.



\subsection{Notation}
Fix throughout an algebraically closed field $\bk$ of characteristic
zero (the algebraically closed hypothesis is for convenience and is
inessential); we will restrict to $\bk=\bC$ in Sections
\ref{s:c-sr}--\ref{s:whc}. We work with algebraic varieties over
$\bk$, which we take to mean 
 reduced separated schemes of finite type over
$\bk$; we frequently work with
affine varieties. (However, see Remarks \ref{r:smooth} and
\ref{r:analytic} for some analogous results in the $C^\infty$ and
complex analytic settings.)
%
We will let $\caO_X$ denote the structure sheaf of $X$, and for
$U \subseteq X$, we let $\caO(U) := \Gamma(U,\caO_X)$.  Recall that
there is a coherent sheaf $T_X$, called the tangent sheaf, on $X$ with
the property that, for every affine open subset $U \subseteq X$,
$\Gamma(U,T_X) \cong \Der(\caO(U))$, the Lie algebra of $\bk$-algebra
derivations of $\caO(U)$, which by definition are the vector fields on
$U$. (In the literature, $T_X$ is often defined as
$\Hom_{\caO_X}(\Omega^1_X, T_X)$ where $\Omega^1_X$ is the sheaf of
K\"ahler differentials. In some references $T_X$ is restricted to the
case that $X$ is smooth, which implies that $T_X$ is a vector bundle,
but in general $T_X$ need not be locally free; see, e.g., \cite[pp.~88--89]{Shaf2v2} for a reference for $T_X$ in general.)  Let
$\Vect(X):=\Gamma(X,T_X)$, which is a Lie algebra whose elements are
called global vector fields on $X$, and which is a module over the
ring $\caO(X)$ of global functions.

\section{Finite-dimensionality of coinvariants under flows and zeroth Poisson homology}
\subsection{} Let $\mathfrak{g}$ be a Lie algebra over $\bk$ and $X$
be an 
affine variety over $\bk$ (which very often will be
singular).  Suppose that $\mathfrak{g}$ acts on $X$, i.e., we have a
Lie algebra homomorphism $\alpha: \mathfrak{g} \to \Vect(X)$ (we can
take $\mathfrak{g} \subseteq \Vect(X)$ and $\alpha$ to be the
inclusion if desired).  In this case, $\mfg$ acts on $\caO(X)$ by
derivations, and we can consider the coinvariant space,
$\caO(X)_\mfg := \caO(X)/\mfg \cdot \caO(X)$, which is also denoted by
$H_0(\mfg, \caO(X))$ (it is the zeroth Lie homology of $\mfg$ with
coefficients in $\caO(X)$).  We begin with a criterion for this to be
finite-dimensional, which is the original motivation for the results
of this note.

Say that the action is transitive if the map $\alpha_x: \mfg \to T_x X$
is surjective for all $x \in X$.  It is easy to see that, if the
action is transitive, then $X$ is smooth, since $\rk \alpha_x$ is
upper semicontinuous, while $\dim T_x X$ is lower semicontinuous (in
other words, generically $\alpha$ has maximal rank and $X$ is smooth, so
if $\alpha_x$ is surjective for all $x$, then $\dim T_x X = \rk \alpha_x$
is constant for all $x$, and hence $X$ is smooth).
\begin{example}
  If $\mfg = \Vect(X)$, then $\mfg$ acts transitively if and only if
  $X$ is smooth (since it is clear that, if $X$ is smooth affine, then
  global vector fields restrict to $T_xX$ at every $x \in X$).  Moreover, if
  $X$ is singular, it is a theorem of A.~Seidenberg \cite{Sei-dirfgt}
  that $\mfg$ is tangent to the set-theoretic singular locus
  (although this would be false in characteristic $p$).
\end{example}
In the case that $\mfg$ does not act transitively, one can attempt to
partition $X$ into leaves where it does act transitively. This motivates
\begin{definition}
  A $\mfg$-leaf on $X$ is a maximal locally closed connected  
subvariety $Z \subseteq X$ such that, for all $z \in Z$, the image
  of $\alpha_z$ is $T_z Z$.
\end{definition}
Note that a $\mfg$-leaf is smooth since the tangent spaces $T_z Z$ have constant dimension, and hence it is irreducible. Thus,
 a $\mfg$-leaf is a maximal locally closed irreducible subvariety
$Z$ such that $\mfg$ preserves $Z$ and acts transitively on it.
\begin{remark} Note that two distinct $\mfg$-leaves are disjoint,
  since if $Z_1, Z_2$ are two intersecting leaves, then the union
  $Z=Z_1 \cup Z_2$ is another connected locally closed set such that
  the image of $\alpha_z$ is $T_z Z$ for all $z \in Z$; by maximality
  $Z_1=Z=Z_2$.  Therefore, if $X$ is a union of $\mfg$-leaves, then this
  union is disjoint and the decomposition is canonical.
 
  On the other hand, it is not always true that $X$ is a union of its
  leaves.  Indeed, $X$ can sometimes have no leaves at all. For
  example, let $X=(\bC^\times)^2=\Spec \bC[x,x^{-1},y,y^{-1}]$ and
  $\mfg = \bC \cdot \xi$ for $\xi$ a global vector field which is not
  algebraically integrable, such as $\xi=x\partial_x - c y \partial_y$
  for $c$ irrational.  If we work instead in the analytic setting,
  then locally there do exist analytic $\mfg$-leaves, which in the
  example are the local level sets of $x^c y$; but these are not algebraic
  (and do not even extend to global analytic leaves), as we are requiring.
\end{remark}
\begin{theorem}\label{t:coinv-fd} (\cite[Theorem 3.1]{ESdm}, 
  \cite[Theorem 1.1]{ES-dmlv}) If $X$ is a union of finitely many
  $\mfg$-leaves, then the coinvariant space $\caO(X)_\mfg = \caO(X) /
  \mfg \cdot \caO(X)$ is finite-dimensional.
\end{theorem}
\begin{remark}
  Let $X_i:=\{x \in X \mid \rk \alpha_x = i\}$ 
  be the locus where the infinitesimal action of $\mfg$ restricts to
  an $i$-dimensional subspace of the tangent space.  This is a locally
  closed subvariety. If it has dimension $i$, then its connected
  components are the leaves of dimension $i$ and there are finitely
  many.  Otherwise, if $X_i$ is nonempty, it has dimension greater
  than $i$ and $X$ is not the union of finitely many $\mfg$-leaves.
  In the case that $\bk=\bC$, there are finitely many analytic leaves
  of dimension $i$ in an analytic neighborhood of every point if and
  only if $X_i$ has dimension $i$ or is empty.  See \cite[Corollary
  2.7]{ES-dmlv} for more details.
\end{remark}
The first corollary (and the original version of the result in
\cite{ESdm}) is the following special case.  Suppose that $X$ is an
affine Poisson variety, i.e., $\caO(X)$ is equipped with a Lie bracket
$\{-,-\}$ satisfying the Leibniz rule, $\{fg,h\}=f\{g,h\} + g\{f,h\}$
(called a Poisson bracket).  Equivalently, $\caO(X)$ is a Lie algebra
such that the adjoint action, $\ad(f):=\{f,-\}$, is by derivations.
We use the notation $\xi_f := \ad(f)$, which is called the Hamiltonian
vector field of $f$. Let $\mfg := \caO(X)$; then we have the action
map $\alpha: \mfg \to \Vect(X)$ given by $\alpha(f) = \xi_f$.  In this
case, the $\mfg$-leaves are called symplectic leaves, because for
every $\mfg$-leaf $Y$ and every $y \in Y$, the tangent space $T_yY$ is
equal to the space of restrictions $\xi_f|_y$ of all Hamiltonian
vector fields $\xi_f$ at $y$.  Then, it is easy (and standard) to see
that the Lie bracket restricts to a well-defined Poisson bracket on
each symplectic leaf, which is nondegenerate, i.e., it induces a
symplectic structure determined by the formula $\omega(\xi_f) = df$.

In this case, the coinvariants $\caO(X)_{\caO(X)}$ are equal to
$\caO(X)/\{\caO(X), \caO(X)\}$, which is the zeroth Poisson homology
of $\caO(X)$ (and also the zeroth Lie homology of $\caO(X)$ with
coefficients in the adjoint representation $\caO(X)$, i.e., the
abelianization of $\caO(X)$ as a Lie algebra). We 
denote it by
$\HP_0(\caO(X))$.
\begin{corollary} \cite[Theorem 3.1]{ESdm} \label{c:hp0-fd}
Suppose that $X$ is Poisson with finitely many symplectic leaves. Then
$\HP_0(\caO(X)) = \caO(X)/\{\caO(X),\caO(X)\}$ is finite-dimensional.
\end{corollary}
\begin{example}\label{ex:finqt-svs}
  Suppose that $V$ is a symplectic vector space and $G < \Sp(V)$ is a
  finite subgroup.  Then, as observed in \cite[\S 7.4]{BrownGordon}, the quotient $X = V/G := \Spec \caO(V)^G$ has
  finitely many symplectic leaves. These leaves can be explicitly
  described: call a subgroup $P < G$ \emph{parabolic} if there exists
  $v \in V$ with stabilizer $P$.  Let $X_P \subseteq X$ be the image
  of vectors in $V$ whose stabilizer is conjugate to $P$.  Then $X_P$ is a
  symplectic leaf. To see this, let $v \in V$ have stabilizer
  $P$. Consider the projection $q: V \to V/G$. Then the kernel of
  $dq|_v$ is precisely $(V^P)^\perp$.  Hence, the differentials
  $d(q^* f)|_v$ for $f \in \caO(V)^G$ form the dual space
  $\omega(V^P, -)$ to $V^P$ at $v$, which means that the Hamiltonian
  vector fields $\xi_{q^* f}$ restrict to $V^P$ at $v$.  Since
  $dq|_{v}(V^P) = T_{q(v)} X_P$, we conclude that $T_{q(v)} X_P$ is
  indeed the space of restrictions of Hamiltonian vector fields $\xi_{f}$, as
  desired.  To conclude that $X_P$ is a symplectic leaf, we have to
  show that it is connected. This follows since it is the image under
  a regular map of a connected set (an open subset of a vector
  space).  As a result, we deduce the following corollary, which was a
  conjecture \cite{AF-FgaPa} of Alev and Farkas.
\end{example}
\begin{corollary}[\cite{BEG}]
  If $V$ is a symplectic vector space and $G < \Sp(V)$ is a finite
  subgroup, then $\HP_0(\caO(V/G))=\caO(V)^G/\{\caO(V)^G,\caO(V)^G\}$ is
  finite-dimensional.
\end{corollary}
In fact, the same result holds if $V$ is not a symplectic vector
space, but a 
symplectic affine variety, using the group
$\SpAut(V)$ of symplectic automorphisms of $V$ (and
$G < \SpAut(V)$ still a finite subgroup).  Again, we conclude that the
symplectic leaves are the connected components of the
 $X_P$ as described above; the same proof
applies except that the kernel of $dq|_v$ for $v \in V$ with
stabilizer $P$ is now $((T_v V)^P)^\perp$ (as we do not trivialize the
bundle $TV$).  Moreover, one has the following more general result:
\begin{corollary}\cite[Corollary 1.3]{ESdm} \label{c:symp-qt}
  If $V$ is a symplectic vector space (or 
symplectic affine variety) and
  $G < \Sp(V)$ (or $\SpAut(V)$) is a finite subgroup, then
  $\caO(V)/\{\caO(V),\caO(V)^G\}$ is finite-dimensional.
\end{corollary}
\begin{remark} \label{r:expl-symp-qt} For $V$ a symplectic affine
  variety over $\bk=\bC$, we can give a more explicit formula for
  $\caO(V)/\{\caO(V),\caO(V)^G\}$ (\cite[Corollary 4.20]{ESdm}), which
  reduces it to the linear case and to some topological cohomology
  groups for local systems on locally closed subvarieties:
\begin{equation}\label{e:expl-symp-qt}
\caO(V)/\{\caO(V),\caO(V)^G\} \cong
\bigoplus_P \bigoplus_{Z \in C_P} H^{\dim Z}(Z, H(TV|_Z/TZ)),
\end{equation}
where $P$ ranges over parabolic subgroups of $G$ (stabilizers of
points of $V$), $C_P$ is the set of connected components of $V^P$, and
$H(TV|_Z/TZ)$ is the topological local system on $Z$ whose fiber at
$z \in Z$ is
$\caO(T_zV/T_zZ) / \{\caO(T_zV/T_zZ),\caO(T_zV,T_zZ)^P\}$, which carries a canonical flat connection by \cite[Proposition 4.17]{ESdm} (induced along any path in $Z$ from any choice of symplectic $P$-equivariant parallel transport along $T_zV/T_zZ$, and the choice will not matter on $H(TV|_Z/TZ)$ by definition).  
\end{remark}
\begin{remark}
  As observed in \cite[Corollary 1.3]{ESdm}, Corollary \ref{c:symp-qt}
  continues to hold (with the same proof) if we only assume that $V$
  is an 
affine Poisson variety with finitely many leaves, and let $G$
  be a finite group acting by Poisson automorphisms.
\end{remark}
\begin{example}\label{ex:symp-res}
  One case of particular interest is that of symplectic
  resolutions. By definition, a resolution of singularities $\rho:
  \widetilde X \to X$ is a symplectic resolution if $X$ is normal and
  $\widetilde X$ admits an algebraic symplectic form, i.e., a global
  nondegenerate closed two-form.  Recall that a resolution of
  singularities is a proper, birational map such that $\widetilde X$ is
  smooth.  In this situation, $X$ is equipped with a canonical Poisson
  structure (fixing the symplectic form on $\widetilde X$): for every open
  affine subset $U \subseteq X$, one has $\caO(U) = \Gamma(U, \caO_X) =
  \Gamma(\rho^{-1}(U), \caO_{\widetilde X})$ since $\rho$ is proper and
  birational and $X$ is normal.  Thus, the Poisson structure on
  $\widetilde X$ gives $\caO(U)$ a Poisson structure, which then gives
  $\caO_X$ and hence $X$ a Poisson structure. Conversely, if we begin
  with a normal Poisson variety $X$, we say that $X$ admits a
  symplectic resolution if such a symplectic resolution $\widetilde X \to
  X$ exists, which recovers the Poisson structure on $X$.

  Then, if $X$ admits a symplectic resolution
  $\rho: \widetilde X \to X$, by \cite[Theorem 2.5]{Kal-ssPpv}, it has finitely many symplectic
  leaves: indeed, for every closed irreducible subvariety $Y \subseteq X$ which is invariant under Hamiltonian flow, if $U \subseteq Y$ is the open dense subset
such that the map $\rho|_{\rho^{-1}(U)}: \rho^{-1}(U) \to U$ is
  generically smooth on every fiber $\rho^{-1}(u), u \in U$, then
  $U$ is an open subset of a leaf. By induction on the dimension of
  $Y$, this shows that $Y$ is a union of finitely many symplectic
  leaves; hence $X$ is a union of finitely many symplectic leaves. See
  \cite{Kal-ssPpv} for details.

  We conclude that, in this case, $\HP_0(\caO(X))=\caO(X)/\{\caO(X),\caO(X)\}$
  is finite-dimensional.
\end{example}
\begin{example}More generally, a variety $X$ is called a
  \emph{symplectic singularity} \cite{Beauss} if it is normal, the
  smooth locus $X_\reg$ carries a symplectic two-form $\omega_\reg$,
  and for any resolution $\rho: \widetilde{X} \to X$,
  $\rho|_{\rho^{-1}(X_\reg)}^* \omega_\reg$ extends to a regular (but
  not necessarily nondegenerate) two-form on $X$ (this condition is
  independent of the choice of resolution, as explained in
  \cite{Beauss}, since a rational differential form an a smooth variety is
  regular if and only if its pullback under a proper birational
  map is regular).
  By \cite[Theorem 2.5]{Kal-ssPpv}, every symplectic singularity has
  finitely many symplectic leaves. Therefore, $\HP_0(\caO(X))$ is finite-dimensional.
\end{example}
\begin{remark} By definition, every variety admitting a symplectic
  resolution is a symplectic singularity. However, the converse is far
  from true.  Let $\bk = \bC$. By \cite[Proposition 2.4]{Beauss}, any
  quotient of a symplectic singularity by a finite group preserving
  the generic symplectic form is still a symplectic singularity. But
  even for a symplectic vector space $V$ it is far from true that
  $V/G$ admits a symplectic resolution for all $G < \Sp(V)$ finite.
  To admit a resolution, by Verbitsky's theorem \cite{Verhsgos}, $G$
  must be generated by symplectic reflections (elements $g \in G$ with
  $\ker(g-\Id) \subseteq V$ having codimension two).  Moreover, a
  series of works \cite{Cohsrc,GordonBaby,Belscms,BSsra,BSsra2} leads
  to the expectation that every quotient $V/G$ admitting a symplectic
  resolution is a product of factors of the form
  $\bC^{2n}/(\Gamma^n \rtimes S_n)$ for $\Gamma < \SL(2,\bC)$ finite,
  or of two exceptional factors of dimension four (by
  \cite{Cohsrc,GordonBaby,Belscms}, this holds at least when $G$
  preserves a Lagrangian subspace $U \subseteq V$ and hence can be
  viewed as a subgroup of $\GL(U)$, and in general by \cite{BSsra2}
  there is a list of cases of groups in dimension $\leq 10$ which
  remain to check).
\end{remark}
\begin{example} \label{ex:ham-red} Let $V$ be a symplectic vector
  space and $G<\Sp(V)$ a reductive subgroup. There is a natural moment
  map, $\mu: V \to \mfg^*$ with $\mfg = \Lie G$, defined as
  follows. Let $\mfsp(V) = \Lie \Sp(V)$, and note
  $\mfsp(V) \cong \Sym^2 V^*$ with the Poisson structure on
  $\Sym V^* \cong \caO(V)$ induced by the symplectic form.  Then the
  moment map $V \to \mfsp(V)^* \cong \Sym^2 V$ is the squaring map,
  $v \mapsto v^2$, and by restriction we get a moment map
  $\mu: V \to \mfg^*$.  We can then define the Hamiltonian reduction,
  $X:= \mu^{-1}(0)/\!/G := \Spec \caO(\mu^{-1}(0))^G$.  This is
  well known to inherit a Poisson structure, which on functions is
  given by the same formula as that for the Poisson bracket on
  $\caO(V)$.  In general, $X$ need not be reduced, but by \cite[\S
  2.3]{Los-Bihm}, the reduced subscheme $X^\red$  has finitely
  many symplectic leaves. These leaves are explicitly given as the
  irreducible components of the locally closed subsets
  $X^\red_P = \{q(x) \mid x \in \mu^{-1}(0), G_x=P, G \cdot x \text{ is closed}\}
  \subseteq X^\red$, where $q: \mu^{-1}(0) \to X$ is the quotient map.
  Therefore, $\HP_0(\caO(V/G))$ is finite-dimensional (since Theorem
  \ref{t:coinv-fd} also applies to Poisson schemes that need not be
  reduced).  Note that this example subsumes Example
  \ref{ex:finqt-svs} (which is the special case where $G$ is finite).
\end{example}
\section{Irreducible representations of quantizations}
\subsection{} We now apply the preceding results to the study of quantizations of Poisson varieties.  Let $X$ be an affine Poisson variety.  Recall the following standard definitions:
\begin{definition}
  A deformation quantization of $X$ is an associative algebra
  $A_\hbar$ over $\bk[\![\hbar]\!]$ of the form $A_\hbar =
  (\caO(X)[\![\hbar]\!], \star)$ where $\caO(X)[\![\hbar]\!] := \{\sum_{m \geq 0} a_m
  \hbar^m \mid a_m \in \caO(X)\}$, and $\star$ is an associative
  $\bk[\![\hbar]\!]$-linear multiplication
   such that $a \star b \equiv ab \pmod \hbar$ and $a \star b
  - b \star a \equiv \hbar \{a,b\} \pmod {\hbar^2}$ for all $a,b \in
  \caO(X)$.
\end{definition}
\begin{remark} Note that the multiplication $\star$ is automatically continuous in the $\hbar$-adic topology, since $(\hbar^m A_\hbar) \star (\hbar^n A_\hbar) \subseteq \hbar^{m+n} A_\hbar$.  
\end{remark}
\begin{definition} If $\caO(X)$ is nonnegatively graded with 
Poisson bracket of degree $-d < 0$, then a
  filtered quantization is a filtered
  associative algebra $A = \bigcup_{m \geq 0} A_{\leq m}$ such that
  $\gr A := \bigoplus_{m \geq 0} A_{\leq m} / A_{\leq m-1} \cong
  \caO(X)$ and such that, for $a \in A_{\leq m}, b \in A_{\leq n}$,
  then $ab-ba \in A_{\leq m+n-d}$ and $\gr_{m+n-d} (ab-ba) = \{\gr_m a,\gr_n b\}$.
\end{definition} 
Given a deformation quantization $A_\hbar$, we consider the
$\bk(\!(\hbar)\!)$-algebra $A_\hbar[\hbar^{-1}]$.
\begin{theorem}\label{t:fin-defq}
  Assume that $X$ is an affine Poisson variety with finitely many
  symplectic leaves.  Then, for every deformation quantization
  $A_\hbar$, there are only finitely many continuous irreducible
  finite-dimensional representations of $A_\hbar[\hbar^{-1}]$.
 If $\caO(X)$ is nonnegatively graded with Poisson
  bracket of degree $-d < 0$ and $A$ is a filtered quantization, then
  there are only finitely many irreducible finite-dimensional
  representations of $A$ (over $\bk$).
\end{theorem}
Here by continuous we mean that the map
$\rho: A_\hbar[\hbar^{-1}] \to \text{Mat}_n(\bk(\!(\hbar)\!))$ is continuous in
the $\hbar$-adic topology, i.e., for some $m \in \bZ$, we have
$\rho(A_\hbar) \subseteq \hbar^m \text{Mat}_n(\bk[\![\hbar]\!])$.  The
basic tool we use is a standard result from Wedderburn theory:
\begin{proposition}\label{p:wedderburn}
If $A$
is an algebra over a field $F$, then the characters (i.e., traces) of
nonisomorphic irreducible finite-dimensional representations of $A$
over $F$ are linearly independent over $F$.
\end{proposition}
\begin{proof}[Proof of Theorem \ref{t:fin-defq}]
  We begin with the second statement. Note that $[A,A]$ is a filtered
  subspace of $A$, and hence $\HH_0(A)=A/[A,A]$ is also 
  filtered. By definition we have
  $\{\caO(X), \caO(X) \} \subseteq \gr [A,A]$. Therefore we obtain a
  surjection
  $\HP_0(\caO(X))=\caO(X)/\{\caO(X),\caO(X)\} \onto \gr \HH_0(A)=\gr
  A/[A,A]$.
  As a result,
  $\dim \HH_0(A) =\dim \gr \HH_0(A) \leq \dim \HP_0(\caO(X))$, which
  is finite by Theorem \ref{t:coinv-fd}.

  Given a finite-dimensional representation $\rho:A \to \End(V)$ of
  $A$, the character $\chi_\rho := \operatorname{tr} \chi$ is a linear
  functional $\chi \in A^*$. As traces annihilate commutators,
  $\chi_\rho \in \HH_0(A)^* = [A,A]^\perp \subseteq A^*$.  By
  Proposition \ref{p:wedderburn}, we conclude that the number of such
  representations cannot exceed $\dim \HH_0(A)^*$.  By the preceding
  paragraph, this is finite-dimensional, so there can only be finitely
  many irreducible finite-dimensional representations of $A$ (at most
  $\dim \HP_0(\caO(X))$).

  For the first statement, the same proof applies, except that now we
  need to take some care with the $\hbar$-adic topology.  Namely, let
  $\overline{[A_\hbar,A_\hbar]}$ be the closure of $[A_\hbar,A_\hbar]$
  in the $\hbar$-adic topology, i.e.,
  $\{\sum_{m \geq 0} \hbar^m c_m \mid c_m \in
  [A_\hbar,A_\hbar]\}$.
  Let $V \subseteq \caO(X)$ be a finite-dimensional subspace such that
  the composition $V \into \caO(X) \onto \HP_0(\caO(X))$ is an
  isomorphism. We claim that
  $V[\![\hbar]\!] \into A_\hbar \onto
  A_\hbar/\hbar^{-1}\overline{[A_\hbar,A_\hbar]}$
  is a surjection.  This follows from the following lemma:
\begin{lemma}\label{l:cont}  $A_\hbar \subseteq V[\![\hbar]\!] + \hbar^{-1}\overline{[A_\hbar,A_\hbar]}$.
\end{lemma}
\begin{proof} We claim that, for every $m \geq 1$,
   $A_\hbar \subseteq V[\![\hbar]\!] + \hbar^{-1}\overline{[A_\hbar,A_\hbar]} +
  \hbar^m A_\hbar$.  We prove it by induction on $m$. For $m=1$ this is true by definition of $V$. Therefore also $\hbar A_\hbar \subseteq V[\![\hbar]\!] + \hbar^{-1}\overline{[A_\hbar,A_\hbar]} +
  \hbar^2 A_\hbar$.  For the inductive step, if $A_\hbar \subseteq V[\![\hbar]\!] + \hbar^{-1}\overline{[A_\hbar,A_\hbar]} +
  \hbar^m A_\hbar$ for $m \geq 1$, then substituting the previous equation into $\hbar^m A_\hbar = \hbar^{m-1} (\hbar A_\hbar)$, we obtain the desired result.

Since $V[\![\hbar]\!]$ and
  $\overline{[A_\hbar,A_\hbar]}$ are closed subspaces of $A_\hbar$,
  it follows that
  $A_\hbar \subseteq V[\![\hbar]\!] + \hbar^{-1} \overline{[A_\hbar,A_\hbar]}$
  which proves the lemma.
\end{proof}  
Next, let $d = \dim \HP_0(\caO(X))$ and suppose that
$\chi_1, \ldots, \chi_{d+1}$ are characters of nonisomorphic
continuous irreducible representations of $A_\hbar[\hbar^{-1}]$ over
$\bk(\!(\hbar)\!)$.  Then there exist
$a_1, \ldots, a_{d+1} \in \bk[\![\hbar]\!]$, not all zero, such that
$\sum_i a_i \chi_i|_{V[\![\hbar]\!]} = 0$, since $\dim V = d < d+1$.
Since the representations were continuous,
$\chi_i(\hbar^{-1}\overline{[A_\hbar,A_\hbar]})=0$ for all $i$. By Lemma
\ref{l:cont}, $\sum_i a_i \chi_i|_{A_\hbar}=0$, and by
$\bk(\!(\hbar)\!)$-linearity, $\sum_i a_i \chi_i = 0$. This again
contradicts Proposition \ref{p:wedderburn}.
\end{proof}
\begin{remark} The proof actually implies the stronger result that
  $A_\hbar \otimes_{\bk[\![\hbar]\!]} K$ has finitely many continuous
  irreducible finite-dimensional representations over $K$ (also at
  most $\dim \HP_0(\caO(X))$), where $K = \overline{\bk(\!(\hbar)\!)}$
  is the algebraic closure of $\bk(\!(\hbar)\!)$ (the field of
  Puiseux series over $\bk$, i.e., $\bigcup_{r \geq 1}
  \bk(\!(\hbar^{1/r})\!)$).
  Here an $n$-dimensional representation $\rho:
A_\hbar \otimes_{\bk[\![\hbar]\!]} K \to \text{Mat}_n(K)$ 
 is continuous if $\rho(A_\hbar) \subseteq \hbar^m \text{Mat}_n(\mathcal{O}_K)$
  for some $m \in \bZ$, where $\mathcal{O}_K=\bigcup_{r \geq 1} \bk[\![\hbar^{1/r}]\!]$ is
  the ring of integers of $K$.
  %
This result is stronger since if $\rho_1,
\rho_2$ are two nonisomorphic irreducible representations of an
algebra $A$
over a field $F$,
then for any extension field $E$,
$\Hom_{A
  \otimes_F E} (\rho_1 \otimes_F E, \rho_2 \otimes_F E) =
\Hom_A(\rho_1, \rho_2) \otimes_F E =
0$, so all irreducible representations occurring over $E$ in $\rho_1
\otimes_F E$ and $\rho_2 \otimes_F E$ are distinct.
\end{remark}
\section{Proof of Theorem \ref{t:coinv-fd} using D-modules}\label{s:proof-t-coinv-fd}
In this section, we explain the proof of Theorem \ref{t:coinv-fd}.  We
need the theory of holonomic D-modules (the necessary definitions
and results are recalled in the appendix; see, e.g., \cite{Hotta} for
more details).  An advantage of using D-modules is that the
approach is local, hence does not essentially require affine
varieties.  However, for simplicity (to avoid, for example, presheaves
of Lie algebras of vector fields), we will explain the theory for
affine varieties, and then indicate how it generalizes.

\subsection{The affine case} The main idea is the following
construction. Given an affine Poisson variety $X$ and a Lie algebra
$\mfg$ acting on $X$ via a map $\alpha: \mfg \to \Vect(X)$, we
construct a D-module $M(X,\mfg)$ which represents the functor of
invariants under the flow of $\mfg$, i.e., such that
$\Hom(M(X,\mfg), N) = N^{\mfg}$ for all D-modules $N$, where we will define $N^{\mfg}$ below.
 Without loss of generality, let us assume
$\mfg \subseteq \Vect(X)$ and that $\alpha$ is the inclusion;
otherwise we replace $\mfg$ by its image $\alpha(\mfg)$.  Let
$i: X \to V$ be any closed embedding into a smooth affine variety $V$.
Denote the ideal of $i(X)$ in $V$ by $I_X$.  Let
$\widetilde \mfg \subseteq \Vect(V)$ be the Lie subalgebra of vector
fields which are tangent to $X$ and whose restriction to $X$ is in
the image of $\alpha$. As recalled in Section \ref{ss:rdmod}, there
are mutually quasi-inverse functors
$i_{\natural}: \DmodX \to \text{mod}_X-\caD(V)$ and
$i^{\natural}: \text{mod}_X-\caD(V) \to \DmodX$ defined in Section
\ref{ss:rdmod}, where $\text{mod}_X-\caD(V)$ denotes the category of
\emph{right} modules over $\caD(V)$ supported on $i(X)$, and $\DmodX$ is
the category of D-modules on $X$; this is in fact the way we
define the category $\DmodX$. (We will call these merely D-modules on $X$,
since using left D-modules gives an equivalent definition: see Remark
\ref{r:lrdmod}.)
\begin{definition}\cite[Definition 2.2]{ESdm}, \cite[2.12]{ES-dmlv}
\label{d:mxg}
$M(X,\mfg) := i^{\natural} \bigl( (\widetilde \mfg
\cdot \caD(V) + I_X \cdot \caD(V)) \setminus \caD(V) \bigr)$.
\end{definition}
We will often work with
$\caO(V)$-coherent
right $\caD(V)$-modules supported on $i(X)$.
Note that, on a smooth variety, such modules are well known to be
vector bundles on $X$ (in more detail, one composes the equivalence between
right and left D-modules
on smooth varieties with the equivalence between $\caO$-coherent
left D-modules on a smooth variety
and vector bundles with flat algebraic connections).

\begin{example}\label{ex:tr}
  Suppose that $\mfg$ acts transitively on $X$; in particular, this
  means $X$ is smooth, so we can take $V=X$. Assume also that $X$ is
  connected.  In this case, by \cite[Proposition 2.36]{ES-dmlv},
  $M(X,\mfg)$ is either a line bundle or zero (this can be shown by a
  straightforward computation of its associated graded module over
  $\caO(T^*X)$, cf.~the proof of Lemma \ref{l:support} below).
  In the case that $\mfg$ preserves a global nonvanishing volume form
  (which is sometimes called an affine Calabi--Yau structure), we
  obtain $\Omega_X$, the canonical right $\caD_X$-module of volume
  forms; the isomorphism $M(X,\mfg) \to \Omega_X$ sends the image of
  $1 \in \caD(X) \onto M(X,\mfg)$ to the nonvanishing volume form.
  This includes the situation where $X$ is symplectic and $\mfg$ is
  either $\caO(X)$ or its image in $\Vect(X)$, the Lie algebra of
  Hamiltonian vector fields on $X$.
\end{example}
Given any D-module $N$ on $X$,
let
$\Gamma_{\caD}(X,N) := \Hom_{\caD(V)}(\caD_X, i_\natural N)$ be the
sections of $N$ supported on $i(X)$ (see Section \ref{ss:dx} for more
details).  For $\xi \in \mfg$, and any lift
$\widetilde \xi \in \widetilde \mfg$,
we have a linear endomorphism of $\Gamma(V,i^{\natural} N)$
given by right multiplication by $\widetilde \xi$.
This
 preserves the linear subspace
$\Gamma_{\caD}(X,N)$. The resulting endomorphism does not depend on the choice of the
lift $\widetilde{\xi}$ and defines a Lie algebra action of $\mfg$
on $\Gamma_{\caD}(X,N)$.
Therefore we may consider the
vector space
$N^\mfg := H^0(\mfg, \Gamma_{\caD}(X,N)) = \{n \in \Gamma_{\caD}(X,N) \mid n \cdot \xi =
0, \forall \xi \in \mfg\}$.
\begin{lemma}\label{l:mxg-defn}
 Definition \ref{d:mxg} does not depend on the choice of
  closed embedding $X \to V$.  Moreover, for every D-module $N$ on
  $X$, we have $\Hom(M(X,\mfg), N) = N^\mfg$.
\end{lemma}
The purpose of the second statement above is to explain what functor
is represented by $M(X,\mfg)$.
\begin{proof}
  The first statement follows from the following alternative
  definition of $M(X,\mfg)$.  Note that $\mfg$ acts on the
  D-module $\caD_X$ (see Section \ref{ss:dx} for its definition)
  on the left by  D-module endomorphisms: there is a
  canonical generator $1 \in \caD_X$, so for any $\xi \in \mfg$ and
  lift $\widetilde \xi \in \widetilde \mfg$, we can define $\xi \cdot 1 =
  \widetilde \xi \in i_{\natural}\caD_X = I_X \cdot \caD_V \setminus \caD_V$,
  and this does not depend on the choice of $\widetilde \xi$. This extends
  uniquely to the claimed action. Then, one may check that $M(X,\mfg)
  = \mfg \cdot \caD_X \setminus \caD_X$.  From this one easily deduces
  the second statement.
\end{proof}
\begin{remark}\label{r:canon-gen}
We see from the proof that there is a canonical surjection $\caD_X \to
M(X,\mfg)$.  Equivalently, there is a canonical global section $1 =
1_{M(X,\mfg)} \in \Gamma_{\caD}(X,M(X,\mfg)) = \Hom(\caD_X, M(X,\mfg))$.  For
every closed embedding $i: X \to V$ into a smooth affine variety, applying
$i_\natural$ to this map and taking the composition $\caD(V) \onto i_\natural \caD_X
\onto i_\natural M(X,\mfg)$, we get a canonical generator $1 \in i_\natural M(X,\mfg)$
as a right module over $\caD(V)$. This is nothing but the image of $1
\in \caD(V)$ under the defining surjection $\caD(V) \to i_\natural
M(X,\mfg)$.  We will make use of this canonical generator below.
\end{remark}

Let $\pi: X \to \pt$ be the projection to a point. We will need the
functor of underived direct image, $H^0\pi_{*}$ (see the appendix for
the definition). Then, we have the following fundamental relationship
between the pushforward to a point of $M(X,\mfg)$ and coinvariants of
$\caO(X)$.
\begin{lemma} \label{l:pi0mx}
  $H^0\pi_{*} M(X,\mfg) \cong \caO(X)_\mfg$. 
\end{lemma}
\begin{proof}
Recall from \eqref{e:dirfl} and \eqref{e:dirdir} that
\begin{multline}
  H^0\pi_{*} M(X,\mfg) = (i_{\natural} M(X,\mfg)) \otimes_{\caD(V)} \caO(V) =
  (\widetilde \mfg \cdot \caD_V + I_X \cdot \caD_V) \setminus \caD_V
  \otimes_{\caD_V} \caO_V \\ \cong \caO_V / (\widetilde \mfg \cdot \caO_V + I_X)
  \cong \caO_X / (\mfg \cdot \caO_X) = (\caO_X)_\mfg. \qedhere
\end{multline}
\end{proof}
The proof of Theorem \ref{t:coinv-fd} rests on an estimate for the
characteristic variety (singular support) of $M(X,\mfg)$ (whose definition we recall in
Definition \ref{d:ss}).  Recall above that a $\mfg$-leaf is smooth.
Therefore, given a closed embedding $X \to V$ into a smooth variety, each
$\mfg$-leaf $Z$ has a well-defined conormal bundle, which we denote by
$T_Z^* V$, which has dimension equal to the dimension of $V$.
\begin{lemma}\label{l:support}
 Suppose $X$ is the union of finitely many $\mfg$-leaves
  and $i: X \to V$ is a closed embedding into a smooth variety. Then the
  characteristic variety of $i_\natural M(X,\mfg)$ is contained in the union of the
  conormal bundles of these $\mfg$-leaves inside $V$.
\end{lemma}
\begin{proof} 
  We make an explicit (and straightforward) computation.  For
  notational convenience, we assume that $X \subseteq V$ and that $i$
  is the inclusion. We equip $i_\natural M(X,\mfg)$ with the good filtration
  given by the canonical generator $1 \in i_\natural M(X,\mfg)$, i.e.,
  $\bigl( i_\natural M(X,\mfg) \bigr)_{\leq m} = (\caD_V)_{\leq m} \cdot 1$;
  see Remark \ref{r:canon-gen}.  Then we claim that the associated
  graded relations of the defining relations, $I_X$ and
  $\widetilde \mfg$, of $i_\natural M(X,\mfg)$, cut out the union of the
  conormal bundles.  In this case the associated graded relations are
  also just $I_X, \widetilde{\mfg} \subseteq \caO(T^*V)$. Then, in
  view of the canonical surjection
  $\caO(T^* V)/(I_X \caO(T^* V) + \widetilde \mfg \caO(T^* V)) \onto
  \gr i_\natural M(X,\mfg)$, we obtain the result.

  The claim follows from a more general one, that does not require the
  assumption that $X$ is the union of finitely many $\mfg$-leaves:
\[
Z(I_X + \widetilde \mfg) = \{(x, p): x \in X, p \in
\img(\alpha_x)^\perp\}.
\]
By definition, the restriction of the RHS to any $\mfg$-leaf is the
conormal bundle to the leaf, which proves the preceding claim.  To
prove the above formula, note first that $I_X \caO(T^*V)$ is nothing
but the ideal $I_{T^* V|_X}$ of the subset $T^* V|_X$, the
restriction of the cotangent bundle to $X \subseteq V$. Then, at each
$x \in X$, the equations $\widetilde \mfg$ cut out, in the cotangent
fiber $T^*_x V$, the perpendicular $\img(\alpha_x)^\perp$.  This
proves the claim, and hence the lemma.
\end{proof}
Since the conormal bundle to a smooth subvariety $Z \subseteq V$ of a
smooth variety $V$ has dimension equal to the dimension of $V$, we
conclude:
\begin{theorem} \label{t:mx-hol}
If $X$ is the union of finitely many $\mfg$-leaves,
  then $M(X,\mfg)$ is holonomic.
\end{theorem}
We remark that this result, in the case that $\mfg$ is the derivative
of the action of a (connected) algebraic group $G$ on $X$, is well-known (see, e.g., \cite[Section 5]{Hot-edm}); see also Remark \ref{r:conj-reg-ss} below.
We can now complete the proof of Theorem \ref{t:coinv-fd}. By Theorem
\ref{t:mx-hol} and Corollary \ref{c:fd-pfwd}, $H^0\pi_{*} M(X,\mfg)$
is finite-dimensional.  By Lemma \ref{l:pi0mx}, this is
$\caO(X)_\mfg$, which is hence finite-dimensional.
\begin{example}\label{ex:m-vg}
Suppose that $X=V/G$ for $V$ a symplectic vector space and $G < \Sp(V)$ a finite
group of symplectic automorphisms.  
For any parabolic subgroup $P$ in $G$ (as in Example \ref{ex:finqt-svs}), let $N(P)$ be the normalizer of
$P$ in $G$, and $N^0(P) := N(P)/P$. Let $i_K: V^P/N^0(P)\to V/G$ be
the corresponding closed embedding. Then by \cite[Corollary 4.16]{ESdm}, 
there is a canonical isomorphism, where $\text{Par}(G)/G$
denotes the set of conjugacy classes $[P]$ of parabolic subgroups $P < G$,
  \begin{equation}\label{e:m-vg}
  M(V/G) \cong \bigoplus_{[P] \in \text{Par}(G)/G} \HP_0(\caO_{(V^P)^\perp/P})\otimes
  (i_K)_*(\IC(V^P/N^0(P))).
  \end{equation}
\end{example}

\subsection{Globalization}\label{ss:glob}
In this subsection, we briefly explain how to generalize the previous
constructions to the not necessarily affine case.  If $X$ is an
arbitrary variety, then we may consider a presheaf $\mfg$ of Lie
algebras acting on $X$ via a map $\alpha: \mfg \to T_X$.  For example, $\mfg$ could be a constant
sheaf, giving a (global) action of $\mfg$ on $X$.  Another example is
if $X$ is a Poisson variety; then $\mfg$ could be $\caO_X$, acting by
the Poisson bracket, or its image in $T_X$, which is the presheaf of
Hamiltonian vector fields.

As before, without loss of generality, let us assume that
$\mfg \subseteq T_X$ is a sub-presheaf and $\alpha$ is the inclusion
(we can just take the image of $\alpha$).  Let $i: X \to V$ be a
closed embedding into a smooth variety $V$ and $\mathcal{I}_X$ the
ideal sheaf of $i(X)$. Let $\widetilde \mfg \subseteq T_V$ be the
sub-presheaf of vector fields which are tangent to $X$ and restrict on
$X$ to vector fields in $\mfg$.  Then, given any open affine subset
$U \subseteq X$, we can consider the D-module $M(U,\mfg(U))$ defined
as in Definition \ref{d:mxg}.  Under mild conditions, these then glue
together to form a D-module on $X$:
\begin{definition}\cite[Definition 3.4]{ES-dmlv}
 The presheaf $\mfg$ is \emph{$\caD$-localizable} if,
for every chain $U' \subseteq U \subseteq X$ of open affine subsets,
\begin{equation}
\mfg(U') \cdot \caD(U') = \mfg(U) \cdot \caD(U').
\end{equation}
\end{definition}
In particular, it is immediate that, if $\mfg$ is a constant sheaf,
then it is $\caD$-localizable. By \cite[Example 3.11]{ES-dmlv}, the
presheaf of Hamiltonian vector fields on an arbitrary Poisson variety
is $\caD$-localizable. By \cite[Example 3.9]{ES-dmlv}, the sheaf of
all vector fields is $\caD$-localizable (note that merely being a
sheaf does not imply $\caD$-localizability, although by \cite[Example
3.10]{ES-dmlv}, being a quasi-coherent sheaf acting in a certain way does imply
$\caD$-localizability).
\begin{proposition}\cite[Proposition 3.5]{ES-dmlv}
If $\mfg$ is $\caD$-localizable, then there is a canonical D-module
$M(X,\mfg)$ on $X$ whose restriction to every open affine $U$ is $M(U,\mfg(U))$.
\end{proposition}
\begin{remark} In \cite{ES-dmlv}, the above is stated somewhat more
  generally: one can fix an open affine covering and ask only that
  $\mfg$ be $\caD$-localizable with respect to this covering.  Above
  we take the covering given by \emph{all} open affine subsets, and it
  turns out that the $\mfg$ we are interested in are all
  $\caD$-localizable with respect to that covering (which is the
  strongest $\caD$-localizability condition).
\end{remark}
With this definition, all of the results and proofs of the previous
section, except for Lemma \ref{l:pi0mx} (and the proof of Theorem
\ref{t:coinv-fd}) carry over. In particular, Theorem \ref{t:mx-hol} holds
for arbitrary (not necessarily affine) $X$.
\begin{example}\label{ex:tr2}
  As in Example \ref{ex:tr}, we can consider the case where $\mfg$
  acts transitively on $X$. Assume $X$ is irreducible. As before, $X$
  is smooth, so $M(X,\mfg)$ is either a line bundle or zero. If
  $\mfg$ preserves a nonvanishing global volume form, we again deduce
  that $M(X,\mfg) = \Omega_X$, by the same isomorphism sending
  $1 \in M(X,\mfg)$ to the volume form. This includes the case that
  $X$ is symplectic (and need not be affine), with the symplectic
  volume form.
\end{example}
\begin{corollary}
If $X$ is a (not necessarily affine) Poisson variety with finitely many
symplectic leaves, and $\mfg$ is $\caO_X$ (or the presheaf of Hamiltonian vector
fields), then $M(X,\mfg)$ is holonomic. 
\end{corollary}
\begin{proof}
  We only need to observe that the $\mfg$-leaves are the symplectic
  leaves. Then the result follows from Theorem \ref{t:mx-hol}.
\end{proof}
\begin{remark}\label{r:presheaves}
  Beware that the use of presheaves above is necessary: for a general
  Poisson variety $X$, the presheaf $\mfg$ of Hamiltonian vector
  fields is \emph{not} a sheaf: see \cite[Remark 3.16]{ES-dmlv} (even
  though $\mfg$ is the image of the action $\alpha: \caO_X \to T_X$ of
  the honest sheaf of Lie algebras $\caO_X$ on $X$). However, as we
  observed there, when $X$ is generically symplectic (which is 
  true when it has finitely many symplectic leaves, as in all of our
  main examples), then $\mfg$ is a sheaf (although it is clearly not
  quasi-coherent).
\end{remark}
\begin{example}
Suppose that $V$ is a symplectic variety (not necessarily affine) and
$G < \SpAut(V)$ a finite subgroup of symplectic automorphisms. Then
we can let $\mfg := H(V)^G$, the $G$-invariant Hamiltonian vector fields.
Then we obtain the following formula (\cite[Theorem 4.19]{ESdm}), which
implies \eqref{e:expl-symp-qt} from before. In the notation of Remark \ref{r:expl-symp-qt}, for $i_Z: Z \to V$ the closed embedding:
\begin{equation}\label{e:expl-symp-qt-dmod}
M(V,\mfg) \cong \bigoplus_P \bigoplus_{Z \in C_P} (i_Z)_* H(TV|_Z/TZ),
\end{equation}
where the sum is over all parabolic subgroups of $V$, and we view
topological local systems on smooth subvarieties of $V$ as left
D-modules and hence right D-modules in the canonical way.

Passing to $V/G$ we get a global generalization of Example
\ref{ex:m-vg} (\cite[Theorem 4.21]{ESdm}): let $\text{Par}(G)/G$ be
the set of conjugacy classes $[P]$ of parabolic subgroups $P < G$, and
for $Z \in C_P$, let $N_Z(P) < N(P)$ be the subgroup of elements of
the normalizer $N(P)$ of $P$ which map $Z$ to itself, and let
$N_Z^0(P) := N_Z(P)/P$.  Let $Z_0 := Z/N^0_Z(K)$ and
$i_{Z_0}: Z_0 \to V/G$ the closed embedding. Let
$\pi_Z: Z \to Z_0$ be the $N_Z^0(P)$-covering, and let
$H_{Z_0} := (\pi_Z)_* H(TV_Z/TZ)^{N_Z^0(P)}$ be the
D-module on $Z_0$ obtained by equivariant pushforward from
$Z$. Then we obtain:
\begin{equation}\label{e:expl-symp-qt-dmod2}
M(V/G) \cong \bigoplus_{[P] \in \text{Par}(G)/G} 
\bigoplus_{Z \in C_P/N(P)} (i_{Z_0})_* H_{Z_0}.
\end{equation}
\end{example}
\begin{remark}\label{r:smooth} In fact, one can prove a $C^\infty$ analogue of \eqref{e:expl-symp-qt-dmod}, using the analogous (but simpler) arguments for distributions rather than  D-modules. Let  $V$ be a compact $C^\infty$-manifold and $G$ be a finite group acting faithfully on $V$.
  For $P \leq G$ a parabolic subgroup and $Z$ a connected component of
  the locus of fixed points of $P$, let $H_Z$ denote the space of flat
  sections of the local system $H(TV|_Z/TZ)$ recalled in Remark
  \ref{r:expl-symp-qt}: by \cite[\S 4.6]{ESdm}, the local system is
  trivial, so $H_Z$ identifies with each fiber:
  $H_Z \cong \caO(T_zV/T_zZ) / \{\caO(T_zV/T_zZ),\caO(T_zV,T_zZ)^P\}$
  for every $z \in Z$. In particular, by Corollary \ref{c:symp-qt},
  $H_Z$ is a finite-dimensional vector space.  Then, \cite[Proposition
  4.23]{ESdm} states that the space of smooth distributions on $V$
  invariant under $G$-invariant Hamiltonian vector fields is
  isomorphic to
\begin{equation}
  \bigoplus_P \bigoplus_{Z \in C_P} H_Z^*.
\end{equation}
In particular it is finite-dimensional and of dimension
$\sum_P \sum_{Z \in C_P} \dim H_Z$.
\end{remark}
\begin{remark}\label{r:analytic} Theorem \ref{t:mx-hol} continues to
  hold in the complex analytic setting, using the results of this
  section. However, the coinvariants $\caO(X)_\mfg$ need not be
  finite-dimensional: for instance, if
  $X = \bC^\times \times (\bC \setminus \bZ)$ equipped with the usual symplectic form from the inclusion $X \subseteq \bC^2$ and $\mfg=\caO(X)$, 
then $\caO(X)_\mfg \cong H^2(X)$, which is infinite-dimensional.
\end{remark}  

\section{Poisson-de Rham and $\mfg$-de Rham homology}
\subsection{} As an application of the constructions of the previous section,
we can define a new derived version of the coinvariants $\caO(X)_\mfg$.
Let $X$ be an affine variety and $\mfg$ a Lie algebra acting on $X$.
\begin{definition}\label{d:mfg-dr}
  The $\mfg$-de Rham homology of $X$, $H^{\mfg-DR}_*(X)$, is defined
  as the full derived pushforward $H^{\mfg-DR}_i(X) := H^{-i}(\pi_*
  M(X,\mfg))$.
\end{definition}
By Lemma \ref{l:pi0mx}, $H^{\mfg-DR}_0(X)=\caO(X)_\mfg$.  In this
case, the pushforward functor $H^0\pi_*$ is right exact, and
$H^{-i}(\pi_*) = L^i (H^0\pi_*)$ is the $i$-th left derived functor
(which is why we negate the index $i$ and define a homology theory,
rather than a cohomology theory).

Using Section \ref{ss:glob}, this definition carries over to the
nonaffine setting, where now $\mfg$ may be an arbitrary
$\caD$-localizable presheaf of vector fields on $X$; however, we no
longer have an (obvious) interpretation of $H^{\mfg-DR}_0(X)$ (and
$H^0\pi_*$ is no longer right exact in general).
\begin{example}\label{ex:poisson}
  In the case that $X$ is Poisson, Definition \ref{d:mfg-dr} defines a
  homology theory which we call the \emph{Poisson-de Rham
    homology}. Recall here that, when $X$ is affine, we let $\mfg$ be
  $\caO(X)$ (or its image in $\Vect(X)$, the Lie algebra of
  Hamiltonian vector fields on $X$).  For general $X$, we let $\mfg$
  be $\caO_X$ (or its image in $T_X$, the presheaf of Lie algebras of
  Hamiltonian vector fields).  We denote this theory by
  $\HP^{DR}_*(X) = H^{\mfg-DR}_*(X)$.  If $X$ is affine, then
  $\HP^{DR}_0(X) = \HP_0(\caO(X))$ is the (ordinary) zeroth Poisson
  homology.  If $X$ is symplectic, then we claim that
  $\HP^{DR}_*(X) = H_{DR}^{\dim X - *}(X)$ is the de Rham cohomology
  of $X$.  Indeed, by Examples \ref{ex:tr} and \ref{ex:tr2}, in this
  case $M(X,\mfg) = \Omega_X$, the canonical right D-module of volume
  forms, and then
  $\HP^{DR}_i(X) = H^{-i} \pi_* \Omega_X = H_{DR}^{\dim X - i}(X)$.
\end{example}
\begin{remark} The Poisson-de Rham homology is quite different, in
  general, from the ordinary Poisson homology. If $X$ is an affine
  symplectic variety, it is true that
  $\HP^{DR}_*(X) \cong \HP_*(\caO(X))$, both producing the de Rham
  cohomology of the variety.  But when $X$ is singular, if $X$ has
  finitely many symplectic leaves, $\HP^{DR}_*(X)$ can be nonzero only in
  degrees $-\dim X \leq * \leq \dim X$, since it is the pushforward of a
  holonomic $\caD$-module on $X$. On the other hand, the ordinary
  Poisson homology $\HP_*(\caO(X))$ is in general nonzero in infinitely many
  degrees if $X$ is singular and affine.
\end{remark}

Theorem \ref{t:mx-hol} (now valid for nonaffine $X$) together
with Corollary \ref{c:fd-pfwd} implies:
\begin{corollary}
  If $X$ is the union of finitely many $\mfg$-leaves, then
  $H^{\mfg-DR}_*(X)$ is finite-dimensional.  In particular, if $X$ is
  a Poisson variety with finitely many symplectic leaves, then
  $\HP^{DR}_*(X)$ is finite-dimensional.
\end{corollary}
\begin{example}
  Suppose $\mfg$ is any Lie algebra (or presheaf of Lie algebras)
  acting transitively on $X$ (in particular, $X$ is smooth), and
  assume $X$ is connected. Then as explained in Example \ref{ex:tr},
  $M(X,\mfg)$, and hence $H^{\mfg-DR}(X)$, is either zero or a line
  bundle
  (i.e., in the complex topology, there exist everywhere local
  sections invariant under (the flow of) $\mfg$, which forms a
  line bundle because such sections are unique up to
  constant multiple).  In the case that $M(X,\mfg)$ is a line bundle, the
  associated left D-module
  $L:=M(X,\mfg) \otimes_{\caO_X} \Omega_X^{-1}$ (for $\Omega_X$ the
  canonical bundle) has a canonical flat connection, and then
  $H^{\mfg-DR}(X) = H_{DR}^{\dim X - i}(X,L)$, the de Rham cohomology
  with coefficients in $L$. 

  This vector bundle with flat connection need not be
  trivial. Consider \cite[Example 2.38]{ES-dmlv}:
  $X = (\bA^1\setminus\{0\}) \times \bA^1 = \Spec \bk[x,x^{-1},y]$
  with $\mfg$  the Lie algebra of vector fields preserving the
  multivalued volume form $d(x^{r}) \wedge dy$ for $r \in \bk$.  It is
  easy to check that this makes sense and that the resulting Lie
  algebra $\mfg$ is transitive.  Then, $M(X,\mfg)$ is the
  D-module whose local homomorphisms to $\Omega_X$ correspond to
  scalar multiples of this volume form, and hence $M(X,\mfg)$ is
  nontrivial (but with regular singularities) when $r$ is not an
  integer.  For $\bk=\bC$, the line bundle $L=M(X,\mfg)^\ell$ with flat connection
has  monodromy $e^{-2 \pi i r}$ going counterclockwise around the unit circle.
\end{example} 

\section{Conjectures on symplectic resolutions}\label{s:c-sr}
\subsection{} In the case where $X$ is symplectic, by Example \ref{ex:poisson},
$\HP_0(\caO(X)) \cong H^{\dim X}(X)$ and
$\HP^{DR}_i(X) \cong H^{\dim X - i}(X)$ for all $i$.  It seems to happen
that, when $\widetilde X \to X$ is a symplectic resolution (see
Example \ref{ex:symp-res}), we can also describe the Poisson-de Rham
homology of $X$ in the same way, and it coincides with that of
$\widetilde X$.  For notational simplicity, let us write
$M(X) := M(X,\caO_X)$ below. In this section, we set $\bk = \bC$.
\begin{conjecture}\label{con:sympres}
If $\rho: \widetilde X \to X$ is a symplectic resolution and $X$
is affine, then:
\begin{itemize}
\item[(a)] $\HP_0(\caO(X)) \cong H^{\dim X}(\widetilde X)$;
\item[(b)] $\HP_i^{DR}(X) \cong H^{\dim X - i}(\widetilde X)$ for all $i$;
\item[(c)] $M(X) \cong \rho_* \Omega_{\widetilde X}$.
\end{itemize}
\end{conjecture}
We remark first that (b) obviously implies (a), setting $i=0$.  Next,
(c) implies (b), since, for $\pi^X: X \to \text{pt}$ and $\pi^{\widetilde X}: \widetilde{X} \to \text{pt}$ the projections to points, we have
$\pi^X \circ \rho = \pi^{\widetilde X}$. Thus, (c)
implies that $\pi^X_* M(X) = \pi^X_* \rho_* \Omega_{\widetilde X} = \pi^{\widetilde X}_* \Omega_{\widetilde X}$, whose cohomology is $H^{\dim \widetilde X - *}(\widetilde X) = H^{\dim X - *}(\widetilde X)$.

Next, note that, if one eliminates (a), the conjecture extends to
the case where $X$ is not necessarily affine.  Indeed, part (c) is a
local statement, so conjecture (c) for affine $X$ implies the same
conjecture for arbitrary $X$ by taking an affine covering.  As before,
(c) implies (b) for arbitrary $X$. 

Since (c) is a local statement, if it holds for
$\rho: \widetilde X \to X$, then it follows that the same statement
holds for slices to every symplectic leaf $Z \subseteq X$. Namely,
recall that the Darboux-Weinstein theorem (\cite{We}; see also
\cite[Proposition 3.3]{Kal-ssPpv}) states that a formal neighborhood
$\hat X_z$ of $z \in Z$, together with its Poisson structure, splits
as a product $\hat Z_z \times X_Z$, for some formal transverse slice
$X_Z$ to $Z$ at $z$, which is unique (and independent of the choice of
$z$) up to formal Poisson isomorphism.
Now, for such a formal slice $X_Z$, letting
$\rho': \widetilde X_Z=\rho^{-1}(X_Z) \to X_Z$ be the restriction of $\rho$,
we obtain a formal symplectic resolution, and then the statement of (c) and hence also of (a) and (b) hold for $\rho'$. In the case that $X_Z$ is the formal neighborhood of the vertex of a cone $C$ (expected to
occur by \cite[Conjecture
1.8]{Kal-gtsr}) and $\rho'$ is the restriction of
 a conical symplectic resolution $\rho_C: \widetilde{C} \to C$, this implies that Conjecture \ref{con:sympres} holds for $\rho_C$ as well.  Here and below a conical symplectic resolution is a $\bC^\times$-equivariant resolution for which
the action on the base contracts to a fixed point (i.e., the base is a cone).

In particular one can use this to compute $\HP_0(\caO(X_Z))$ and
$\HP_i^{DR}(X_Z)$ for all leaves $Z$.  Conversely, \cite[Theorem
4.1]{PS-pdrhhvnc} shows that, if one can establish the formal analogue
of (a) for all $X_Z$, and if $\rho: \widetilde X \to X$ itself is a
conical symplectic resolution, 
then Conjecture (c) follows for $X$.
\begin{remark}\label{r:conj-ss-qt}
 Since $\rho$ is  semismall (\cite[Lemma
  2.11]{Kal-ssPpv}), it follows from the decomposition theorem
  \cite[Th\'eor\`eme 6.2.5]{BBDfp} that $\rho_* \Omega_{\widetilde X}$
  is  a semisimple regular holonomic  D-module on
  $X$. Moreover, by \cite[Proposition 2.1]{PS-pdrhhvnc},
  $\rho_* \Omega_{\widetilde X}$ is isomorphic to the semisimplification
 of a quotient $M(X)'$ of
  $M(X)$.
  Conjecture \ref{con:sympres} therefore states that
  $M(X) \cong M(X)'$ and that $M(X)'$ is semisimple. In the case
  that $\rho$ is a conical symplectic resolution, by \cite[Proposition
  3.6]{PS-pdrhhvnc}, $\rho_* \Omega_{\widetilde X}$ is actually rigid,
  which implies that any D-module whose semisimplification is
  $\rho_* \Omega_{\widetilde X}$ is already semisimple.  Thus, in the
  conical case, $M(X)' \cong \rho_* \Omega_{\widetilde X}$, and
  Conjecture \ref{con:sympres} states that in fact the quotient
  $M(X) \onto M(X)'$ is an isomorphism.  For more details on the
  quotient $M(X)'$, see Remark \ref{r:conj-flat} below.
\end{remark}
Conjecture \ref{con:sympres} has been proved in many cases, with the
notable exception of Nakajima quiver varieties.
\begin{remark}
  Let $X$ be a Nakajima quiver variety.  By \cite[Theorem
  4.1]{PS-pdrhhvnc}, since formal slices to all symplectic leaves of
  $X$ are formal neighborhoods of Nakajima quiver varieties, to prove
  Conjecture \ref{con:sympres} for $X$, it would suffice to prove part
  (a) for $X$ and for all the quiver varieties that appear by taking
  slices.
  Thus the full conjecture for the class of Nakajima quiver varieties
  would follow from part (a) for the class of Nakajima quiver
  varieties.
\end{remark}
\begin{example}
  Let $Y$ be a smooth symplectic surface.  Then one can set $X =
  \Sym^n Y := Y^n / S_n$, the $n$-th symmetric power of $Y$.  In this
  case one has the resolution $\rho: \widetilde X = \Hilb^n Y \to X$.  In
  this case, Conjecture \ref{con:sympres}(c) (and hence the entire conjecture) follows
  from \cite[Theorem 1.17]{hp0weyl}, which gives a direct computation
  of $M(X)$ in this case; see Section \ref{s:hp0weyl} for more details.
\end{example}
\begin{example} Next suppose that $Y = \bC^2/\Gamma$, for $\Gamma <
  \SL(2,\bC)$ finite, is a du Val singularity, and $X := \Sym^n Y$.
  Then we can take the minimal resolution $\widetilde Y \to Y$.  We obtain
  from the previous example the resolution $\rho_1: \widetilde X :=
  \Hilb^n \widetilde Y \to \Sym^n \widetilde Y$, and we can compose this with
  $\rho_2: \Sym^n \widetilde Y \to \Sym^n Y$ to obtain the resolution
  $\rho = \rho_2 \circ \rho_1: \widetilde X \to X$.  In this case,
  Conjecture \ref{con:sympres} is proved in \cite[\S 1.3]{hp0weyl}, using the main
  result of \cite{ESsym} together with \cite[Theorem 3]{GoSohilb}: see Section \ref{s:hp0weyl} for more details.
\end{example}
\begin{example}\label{ex:nilcone}
  Suppose that $X$ is the cone of nilpotent elements in a complex
  semisimple Lie algebra $\mfg$. Let $\mathcal{B}$ be the flag
  variety, parameterizing Borel subalgebras of $\mfg$.  The cotangent
  fiber $T^*_{\mfb} \mathcal{B}$ identifies with the annihilator of
  $\mfb$ under the Killing form, i.e., the nilradical $[\mfb,\mfb]$.
  Then one has the Springer resolution $\rho: T^* \mathcal{B} \to X$,
  given by $\rho(\mathfrak{b},x) = x$.  In this case, Conjecture \ref{con:sympres}
  is a consequence of \cite[Theorem 4.2 and Proposition
  4.8.1(2)]{HKihs} (see also \cite[\S 7]{LSssi}), as observed in
  \cite{ESwalg}.   
\end{example}
\begin{example}\label{ex:cl-walg}
  Let $X$ be as in the previous example.  In this case the symplectic
  leaves are the nilpotent adjoint orbits $G \cdot e \subseteq X$, for
  $G$ a semisimple complex Lie group with $\Lie G = \mfg$.  Let
  $e \in \mfg$ be a nilpotent element. One has a Kostant--Slodowy slice
  $S^0_e := X_{G \cdot e}$, transverse to $G \cdot e$, explicitly
  given by $S_e := e + \ker(\ad f)$ and $S_e^0 := S_e \cap X$, where
  $(e,h,f)$ is an $\mathfrak{sl}_2$-triple (whose existence is
  guaranteed by the Jacobson-Morozov theorem), equipped with a
  canonical Poisson structure, such that the formal neighborhood
  $\hat S_e$ of $e$ is a formal slice.  Let $\widetilde S_e^0$ be the
  preimage of $S^0_e$ under $\rho$.  The Poisson algebra $\caO(S_e^0)$
  is called a \emph{classical $W$-algebra}.

  As observed above, the previous example implies that Conjecture
  \ref{con:sympres} also holds for
  $\rho|_{\widetilde S_e^0}: \widetilde S_e^0 \to S_{e}^0$.  In
  particular, one concludes that
  $\HP_0(\caO(S_e^0)) \cong H^{\dim S_e^0}(\widetilde S_e^0)$
  (\cite[Theorem 1.6]{ESwalg}), and the latter is the same as the top
  cohomology of the Springer fiber over $e$,
  $H^{\dim \rho^{-1}(e)}(\rho^{-1}(e))$, since $S_e$ and hence $S_e^0$
  admits a contracting $\bC^\times$ action (the Kazhdan action) to
  $\rho^{-1}(e)$ (explicitly, this is given on $S_e$ by
  $\lambda \cdot (x) = \lambda^{2-h} x$).

  More generally, we conclude also part (b) of Conjecture \ref{con:sympres} for
  $S_e^0$, which yields in this case that $\HP_i^{DR}(S_e^0) \cong
  H^{\dim \rho^{-1}(e)-i}(\rho^{-1}(e))$. By \cite[Theorem
  1.13]{ESwalg}, we can generalize even further and consider all of
  $S_e = e + \ker(\ad f)$, and show that $\HP_*^{DR}(S_e)$ is a
  (graded) vector bundle over $\mfg/\!/G \cong \mfh/W$ with fibers given by the
  cohomology of the Springer fiber over $e$. Equivalently, we have the
  family of deformations $S_e^\eta := S_e \cap \chi^{-1}(e)$ over
  $\mfg/\!/G$, with $\chi: \mfg \to \mfg/\!/G$ the quotient; then we
  conclude that $HP_i^{DR}(S_e^\eta)$ are all isomorphic to $H^{\dim
    \rho^{-1}(e)-i}(\rho^{-1}(e))$, for all $\eta \in \mfg/\!/G$ (the
  family $\HP_i^{DR}(S_e^\eta)$ is flat over $\mfg/\!/G$). 
\end{example}
\begin{remark} \label{r:conj-flat} The deformation considered in
  Example \ref{ex:cl-walg} is part of a more general phenomenon. For a
  general projective symplectic resolution $\rho: \widetilde X \to X$, 
  in \cite{Kal-deq}, Kaledin proves that $\rho$ can be extended
to  a projective map
  $\rho:\widetilde{\mathcal{X}}\to\mathcal{X}$ of schemes over the
  formal disk $\Delta := \Spec \bC[[t]]$, such that, restricting to
  the point $0 \in \Delta$, we recover the original resolution
  $\rho: \widetilde{X} \to X$.
  Furthermore, he shows that $\mathcal{X}$ is normal and flat over
  $\Delta$, and that over the generic point, $\rho$ restricts to an
  isomorphism of smooth, affine, symplectic varieties \cite[2.2 and
  2.5]{Kal-deq}. (In fact the construction provides more: for every
  choice of ample line bundle $L$ on $\widetilde{X}$, there is a unique
  triple $(\widetilde{\mathcal{X}},\mathcal{L},\omega_Z)$ up to
isomorphism
where $\mathcal{L}$ is a line bundle on $\widetilde{\mathcal{X}}$,
$\omega_Z$ is a symplectic structure on the associated
$\bC^\times$-torsor $Z \to \widetilde{\mathcal{X}}$, the $\bC^\times$
action is Hamiltonian for $\omega_Z$, the projection $Z \to \Delta$ is
the moment map for the action, and the restrictions of $\mathcal{L}$
and $\omega_Z$ to $\widetilde X$ recover $L$ and the original
symplectic structure.)  The family of maps over $\Delta$ (together
with $\mathcal{L}$ and $\omega_Z$) is called a {\bf twistor
  deformation}.  In the case that $\rho$ is conical, we can moreover
replace the formal disc $\Delta$ by the line $\bC=\Spec \bC[t]$ and
the map $\widetilde{\mathcal{X}} \to \mathcal{X}$ can be taken to be
$\bC^\times$-equivariant.

In the general case (where $\rho$ need not be conical), let $X_t$ be
the fiber of $\mathcal{X} \to \Delta$ over $t \in \Delta$ and
$\widetilde{X_t}$ the fiber of $\widetilde{\mathcal{X}} \to \Delta$
over $t$.
Then one can show that Conjecture
  \ref{con:sympres} implies that the family $\HP_{DR}^i(X_t)$ is flat
  with fibers isomorphic to $H^{\dim X-i}(\widetilde {X_t})$ (which is
  a vector bundle equipped with the Gauss-Manin connection). More
  generally, the conjecture for $X$ implies that the family $M(X_t)$
  of fiberwise D-modules is torsion-free: indeed, as explained in
  \cite[Proposition 2.1]{PS-pdrhhvnc}, since
  $M(X_t) \cong \Omega_{X_t} \cong \rho_* \Omega_{\widetilde {X_t}}$
  for generic $t$, the quotient $M(X)'$ of $M(X)$ by the torsion of
  the family $M(X_t)$ is isomorphic to the semisimplification of
  $\rho_* \Omega_{\widetilde X}$.  Conversely, if $\rho$ is
  conical, then as explained in Remark \ref{r:conj-ss-qt}, 
  we can replace the formal deformation by an actual $\bC^\times$-equivariant
deformation over the line $\bC$, and in this case
  \cite[Proposition 3.6]{PS-pdrhhvnc} implies that, if the family
  $M(X_t)$ is torsion-free, then $M(X)$ is already semisimple and
  the conjecture holds.
\end{remark}

\begin{example}\label{ex:hypertoric}
  Next suppose that $X$ is a conical Hamiltonian reduction of a
  symplectic vector space by a torus.  Such a variety is called a
  hypertoric cone.  More precisely, we can assume the
  symplectic vector space is a cotangent bundle, $V = T^* U$, and the
  torus is $G = (\bC^\times)^k$  for some $k \geq 1$, acting faithfully
  on $U$ via $a: G \to \GL(U)$, with the induced Hamiltonian action on
  $V$ as in Example \ref{ex:ham-red}. Explicitly if $U=\bC^n$ for
  $n \geq k$, $a(G)$ is a subgroup of the group of invertible diagonal matrices, and $(a_{ij})$ is the matrix of weights such that
  $a(\lambda_1,\ldots,\lambda_k)(e_i) = \prod_{j=1}^k \lambda^{a_{ij}}
  e_i$,
  then
  $\mu((b_1,\ldots,b_n),(c_1,\ldots,c_n))= (\sum_{i=1}^n a_{ij} b_i
  c_i)_{j=1}^k$.
  Then, $X = \mu^{-1}(0)/\!/G$.  In this case, for every character
  $\chi$ of $G$, we can form a GIT quotient
  $\widetilde X := \mu^{-1}(0)/\!/_\chi(G)$, mapping projectively to
  $X$.  In the case this is a symplectic resolution, Conjecture
  \ref{con:sympres} is proved in \cite[Theorem 4.1, Example
  4.6]{PS-pdrhhvnc}, by showing, as we mentioned, that Conjecture
  \ref{con:sympres} follows (for conical symplectic resolutions) from
  its part (a) for slices to the symplectic leaves; since the slices
  in this case are also hypertoric cones, part (a) follows for these
  by \cite{Pr12}.
\end{example}
\begin{remark}\label{r:conj-reg-ss}
  As noted in Remark \ref{r:conj-ss-qt}, Conjecture \ref{con:sympres}
  would imply that $M(X)$ is regular and semisimple when $X$ admits a
  symplectic resolution.  This is not true for general $X$: we will
  explain in Remark \ref{r:mx-nonss-cone} that already if $X$ is a
  surface in $\bC^3$ which is a cone over a smooth curve in $\bP^2$,
  then $M(X)$ is not semisimple unless the genus is zero (hence $X$ is
  a quadric surface).

  For regularity, \cite[Example 4.11]{ESdm} gives a simple example
  where $M(X)$ is not regular: let $X = Z \times \bC^2$ with $Z$ is
  the surface $x_1^3+x_2^3+x_3^3=0$ in $\bC^3=\Spec
  \bC[x_1,x_2,x_3]$.
  Using coordinates $p,q$ on $\bC^2$, we consider the Poisson bracket
  given by $\{p,q\}=1$, $\{x_1,x_2\}=x_3^2$ (and cyclic permutations),
  and $\{q,f\}=0$, $\{p,f\}=|f|f$ for homogeneous $f \in \caO(Z)$ of
  degree $|f|$.  Then $X$ has two symplectic leaves:
  $X \setminus (\{0\} \times \bC^2)$ and $\{0\} \times \bC^2$.  Now
  $\HP_0(\caO(Z))$ is a graded vector space (under
  $|x_1|=|x_2|=|x_3|=1$) of the form
  $\HP_0(\caO(Z)) \cong \bC \oplus \bC^3[-1] \oplus \bC^3[-2] \oplus
  \bC[-3]$
  (a basis is given from monomials in $x_1,x_2,x_3$ of degree at most
  one in each variable).  As a result, the algebraic flat connections
  on $\{0\} \times \bC^2$ given by $\nabla(f) = df - mf dp $,
  $m \in \{0,1,2,3\}$, all appear as quotients of $M(X)$ (i.e., they
  admit sections which extend to Hamiltonian-invariant distributions
  on $X$ supported on $\{0\} \times \bC^2$). As these connections have
  irregular singularities at $\{\infty\} = \bP^2 \setminus \bC^2$ for
  $m\neq 0$, we conclude that $M(X)$ is not regular.

  However, it is an open question whether, if $X$ has finitely many symplectic leaves,  $M(X)$ must  be
  \emph{locally regular} on $X$, i.e., all composition factors are
  rational connections which have no irregular singularities in $X$
  itself.  If this is true, then it would follow that $M(X)$ is
  regular whenever $X$ is proper (in particular, projective).  

  Let us remark that in the case when $\mathfrak{g} = \Lie G$ and the
  action is the infinitesimal action associated to an action of $G$ on
  $X$ with finitely many orbits, then it is well-known that $M(X)$ is
  regular holonomic (see, e.g., \cite[Section 5]{Hot-edm}). Note that,
  for general $\mathfrak{g}$, the action may not integrate to a group
  action, but formally locally it integrates to the action of a formal
  group; it would be interesting to try to use this and the argument
  of \emph{op.~cit.} to prove local regularity in general.
\end{remark}
\begin{remark} There are many other interesting consequences of
  Conjecture \ref{con:sympres} which would resolve open questions.
  For instance, the conjecture implies that every symplectic
  resolution of $X$ is strictly semismall in the following sense: for
  every symplectic resolution $\rho: \widetilde X \to X$ and every
  symplectic leaf $Y \subseteq X$, one has
  $\dim \rho^{-1}(Y) = \frac{1}{2}(\dim X + \dim Y)$.  The
  semismallness condition itself is equivalent to the inequality
  $\leq$.  This corollary follows because, whenever $X$ has finitely
  many symplectic leaves, the intersection cohomology D-module of
  every symplectic leaf closure (i.e., intermediate extension of the
  canonical right D-module on the leaf itself) is a composition factor
  of $M(X)$ (which follows from \cite[\S 4.3]{ESdm}; see also
  \cite[Propositions 2.14 and 2.24]{ES-dmlv}), and there is a
  composition factor of $\rho_* \Omega_{\widetilde X}$ with support
  equal to the leaf closure if and only if the dimension equality
  holds.  Another interesting potential application (pointed out to us
  by D.~Kaledin) is a conjecture variously attributed to Demailly,
  Campana, and Peternell \cite[Conjecture 1.3]{Kal-gtsr} that, if
  $T^*Z \to Y$ is a symplectic resolution of an affine variety $Y$,
  then $Z$ is a partial flag variety. Namely the conjecture implies
  that the maximal ideal $\mathfrak{m}_0 \subseteq \caO(Y)$ of the
  origin is a perfect Lie algebra; the conjecture would follow if one
  shows that $Z = G/P$ where $\Lie G \subseteq \mathfrak{m}_0$ is the
  degree-one subspace and $P$ is a parabolic subgroup of $G$.
\end{remark}

\section{Symmetric powers and Hilbert schemes}\label{s:hp0weyl}

In this section we would like to discuss results from \cite{hp0weyl}
on the zeroth Poisson homology of symmetric powers. We continue to set $\bk=\bC$.  In this section the affine variety $Y$ will always be assumed to be connected.

\subsection{The main results} \label{ss:hp0weyl-main}

Given an affine variety $Y = \Spec A$,  let $S^n Y := Y^n / S_n = \Spec \Sym^n A$ be the $n$-th symmetric
power of $Y$.  Let the symbol $\&$ denote the product in
the symmetric algebra. Note that $\bigoplus_{n \geq 0} \HP_0(\caO(S^n Y))^*$ is a graded algebra, with multiplication induced, via the inclusions $\HP_0(\caO(X))^* \subseteq \caO(X)^*$, by the maps $\caO(S^m Y)^* \otimes \caO(S^n Y)^* \to \caO(S^{m+n} Y)^*$ dual to the symmetrization maps $\caO(S^{m+n}Y) \to \caO(S^m Y) \otimes \caO(S^n Y)$ sending $f$ to the function
\[
((x_1,\ldots,x_m),(x_{m+1},\ldots,x_{m+n})) \mapsto \frac{1}{(m+n)!} \sum_{\sigma \in S_{m+n}} f(x_{\sigma(1)},\ldots,x_{\sigma(m+n)}).
\]
To see that this indeed induces maps on Poisson traces ($\HP_0^*$),
note that $\caO(Y)$ acts on $\caO(S^n Y) = \Sym^n \caO(Y)$ by Lie
bracket, and
$\HP_0(\caO(S^n Y))^* = (\caO(S^n Y)^*)^{\caO(S^n Y)} = (\caO(S^n
Y)^*)^{\caO(Y)}$.
Then it remains to observe that the maps above are compatible with
this adjoint action of $\caO(Y)$, so they indeed induce bilinear maps
as claimed on Poisson traces, which are easily seen to be associative
with unit $1 \in \HP_0(\caO(S^0Y))=\HP_0(\bC)=\bC$.
\begin{theorem}\label{symphp0thm} (\cite{hp0weyl}, Theorem 1.1)
  Let $Y$ be an affine symplectic variety.  Then, there is a canonical
  isomorphism of graded algebras,
\begin{gather} \label{e:symphp0thm}
  \Sym(\HP_0(\caO(Y))^*[t]) \iso \bigoplus_{n \geq 0} \HP_0(\caO(S^n
    Y))^*, \\ \notag \phi \cdot t^{m-1} \mapsto \Bigl( (f_1 \&
  \cdots \& f_m) \mapsto \phi(f_1 \cdots f_m) \Bigr),
\end{gather}
where the grading is given by $|\HP_0(\caO(S^n Y))^*| = n$ (on both
sides of the isomorphism), and $|t| = 1$.
\end{theorem}

If we expand the symmetric algebra on the LHS in \eqref{e:symphp0thm} and dualize, we explicitly obtain the following. Recall that a partition of $n$ of length $k$ is a tuple $\lambda = (\lambda_1, \ldots, \lambda_k)$ such that $\lambda_1 \geq \lambda_2 \geq \cdots \geq \lambda_k \geq 1$ and $\lambda_1 + \cdots + \lambda_k = n$. If $\lambda$ is a partition of $n$, we write $\lambda \vdash n$, and let $|\lambda|$ denote
its length. Let $S_\lambda <
S_{|\lambda|}$ be the subgroup preserving the partition $\lambda$.  Explicitly,
$S_{\lambda} = S_{r_1} \times \cdots \times S_{r_k}$ where, for all
$j$,
\[
\lambda_{r_1+\cdots+r_j} > \lambda_{r_1+\cdots+r_j+1} = \lambda_{r_1 +
  \cdots + r_j + 2} = \cdots = \lambda_{r_1 + \cdots + r_j + r_{j+1}}.
\]
Then \eqref{e:symphp0thm} states that, for all $n \geq 1$,
\begin{equation}
\HP_0(\caO(S^n Y)) \cong \bigoplus_{\lambda \vdash n} (\HP_0(\caO(Y))^{\otimes |\lambda|})_{S_\lambda}.
\end{equation}
Next, it is well known that, if $Y$
is connected, then $\HP_0(\caO(Y)) \cong H^{\dim Y}(Y)$, the top
cohomology of $Y$, via the isomorphism $[f] \mapsto f \cdot
\operatorname{vol}_Y$, where $\operatorname{vol}_Y$ is the canonical
volume form (i.e., the $\frac{1}{2} \dim Y$-th exterior power of the
symplectic form). We can write the above more explicitly using the
coefficients $a_n(i)$ which give the number of $i$-multipartitions
of $n$ (i.e., collections of $i$ ordered partitions whose sum of sizes is $n$), i.e.,
\begin{equation}\label{e:ani-defn}
\prod_{m \geq 1} \frac{1}{(1-t^m)^i} = \sum_{n \geq 0} a_n(i) \cdot t^n.
\end{equation}

\begin{corollary}\label{c:mpfla}(\cite{hp0weyl}, Corollary 1.2) If $Y$ is a 
symplectic variety, then 
$\dim \HP_0(\caO(S^n Y)) = a_n(\dim H^{\dim Y}(Y))$.
\end{corollary}

\begin{notation} Whenever we take tensor products (so also symmetric
  powers) of $\bC[\![\hbar]\!]$-algebras complete in the $\hbar$-adic
  topology (e.g., $\Sym^n A_\hbar$), we mean the $\hbar$-adic
  completion of the usual tensor product (and hence symmetric power).
\end{notation}
\begin{notation}
When $B$ is a $\bC[\![\hbar]\!]$-algebra complete in the $\hbar$-adic topology, let  $\HH_0(B) := B / \overline{[B,B]}$ (i.e., we
 take the closure, equivalently $\hbar$-adic completion, of $[B,B]$).
\end{notation}
The results above imply the degeneration of the spectral sequence
computing the zeroth Hochschild homology of quantizations of $S^nY$.
Let $Y$ be an affine symplectic variety, and let $A_\hbar$ be any
deformation quantization of $\caO(Y)$, so that $\Sym^n A_\hbar$ is a
deformation quantization of $\caO(S^n Y)$).  Then the spectral
sequence associated to the deformation yields a natural
$\mathbb{C}(\!(\hbar)\!)$-linear surjection
 $$
 \theta: \HP_0(\caO(S^n Y))(\!(\hbar)\!)
 \onto \gr 
\HH_0(\Sym^n A_\hbar[\hbar^{-1}]),
 $$
 where the filtration on $\HH_0(\Sym^n A_\hbar[\hbar^{-1}])$ is
 induced by the filtration of $\Sym^n A_\hbar[\hbar^{-1}]$ by powers
 of $\hbar$, and $\gr$ denotes the
 $\hbar$-adically completed
associated graded space.

\begin{corollary}\label{defcor1} (\cite{hp0weyl}, Corollary 1.3) 
$\theta$ is an isomorphism. 
\end{corollary}

Namely, Corollary \ref{defcor1} follows from Corollary \ref{c:mpfla} 
and the computation of $\HH_0(\Sym^nA_\hbar[\hbar^{-1}])$ from \cite{EO}, which jointly show that 
$\dim \HP_0(\mathcal O(S^nY))=\dim \HH_0(\Sym^nA_\hbar[\hbar^{-1}])$. 

\subsection{Sketch of proof of Theorem \ref{symphp0thm}}
The proof of Theorem \ref{symphp0thm} is based on the following
theorem, giving the structure of $M(X)$ when $X = S^n Y$, for $Y$ a
symplectic variety that need not be affine.  Let
$\Delta_i: Y \into S^i Y$ be the diagonal embedding, and for
$\sum_{j=1}^k r_j i_j = n$, let
$q: (S^{i_1} Y)^{r_1} \times \cdots \times (S^{i_k} Y)^{r_k} \onto S^n
Y$ be the obvious projection.

\begin{theorem}\label{sympdmthm} (\cite{hp0weyl}, Theorem 1.17)
 \begin{equation} \label{e:sympdmthm}
M(S^n Y) \cong \bigoplus_{r_1 \cdot i_1 + \cdots + r_k \cdot i_k = n, 1 \leq i_1 < \cdots < i_k, r_j \geq 1\, \forall j}
  q_* \bigl((\Delta_{i_1})_*(\Omega_Y)^{\boxtimes r_1} \boxtimes \cdots \boxtimes
  (\Delta_{i_k})_*(\Omega_Y)^{\boxtimes r_k} \bigr)^{S_{r_1} \times \cdots \times S_{r_k}}.
\end{equation}
\end{theorem}

Indeed, Theorem \ref{symphp0thm} (at the level of vector spaces) is
obtained from Theorem \ref{sympdmthm} by computing the direct image of
$M(X)$ to the point, and it is not hard to check that the
corresponding isomorphism of vector spaces is an algebra map.

\subsection{Sketch of proof of Theorem \ref{sympdmthm}} 
Now let us say a few words about the proof of Theorem
\ref{sympdmthm}. One can show that all the simple summands $S$ on the
right hand side are composition factors of $M(X)$ (by constructing
surjections $M(X)|_U\to S|_U$ on dense open sets $U$ in $X$), so it
suffices to show that: (a) they are the only composition factors, which
furthermore occur with multiplicity $1$, and (b) $M(X)$ is
semisimple.  

We first show that (a) implies (b). To do this we need to work on
$Y^n$ rather than on $X=S^n Y = Y^n/S_n$.  Let $p: Y^n \to S^n Y=X$ be
the projection. By definition, $M(X) \cong p_* (\widetilde M)^{S_n}$
where $\widetilde M := M(Y^n,\caO(Y^n)^{S_n})$ is a right D-module
on $Y^n$, which is also holonomic since the $\caO(Y^n)^{S_n}$-leaves of $Y^n$ are the diagonals (subvarieties $Z \subseteq Y^n$ obtained by setting certain components to be equal).
  Similarly, the summands $S$ in Theorem \ref{sympdmthm} are
of the form $S \cong p_*(\widetilde S)^{S_n}$ for some D-modules
$\widetilde S$ on $Y^n$. In fact, each $\widetilde S$ is the pushforward
of the canonical
right D-module $\Omega_{Z}$ for some diagonal 
 $Z \cong Y^m \subseteq Y^n, m \leq n$ under the embedding $Z \to Y^n$.  
Since the $\widetilde S$ are also composition factors of
$\widetilde M$,  to deduce (b) from (a), it suffices to
show that $\Ext^1(\widetilde S, \widetilde{S'})=0$ for distinct
$S,S'$.  The characteristic variety of each $\widetilde S$ is the conormal
bundle $T^*_Z Y$ of the associated diagonal $Z \subseteq Y^n$, and for
distinct diagonals the intersection of these conormal bundles has
codimension at least $\dim Y \geq 2$ (i.e., the dimension is at most
$(n-1)\dim Y$).  By a well-known result from D-module theory
(\cite[Theorem 1.2.2]{KK-hsmde3}, see also \cite[1.4]{KV-mdsc3}), this
implies that $\Ext^1(\widetilde S, \widetilde {S'}) = 0$. (For a
slightly different argument avoiding \cite[Theorem 1.2.2]{KK-hsmde3},
see
\cite{hp0weyl}, Lemma 2.1).

To prove (a),
it suffices to replace $S^nY$ with the formal neighborhood of a point
of the diagonal in $S^nY$. In other words, by the formal Darboux
theorem, it is sufficient to consider the flat case, when
$Y=\widehat V$ is the formal neighborhood of zero in a symplectic
vector space $V$.  In this case, by Example \ref{ex:m-vg}, we only have
to show that each multiplicity space for
the intersection cohomology D-module of each diagonal is
 one-dimensional for all $m \leq n$.  By induction
on $n$ we can restrict to the delta function D-module of the origin.
Then since $M(S^n V)$ is semisimple it suffices to show that
\begin{equation}\label{eqq}
\Hom_{\mathcal D(S^n Y)}(M(S^n Y), \delta_{Y}) \cong
\bC.
\end{equation}
Finally, \eqref{eqq} can be restated without using
$\mathcal D$-modules in the form of the following lemma, which plays a
central role in the proof, and concludes our sketch of it:

\begin{lemma} (\cite{hp0weyl}, Lemma 2.3) \label{l:hidist} The space
  of symmetric polydifferential operators
  $\psi: \mathcal{O}(V)^{\otimes (n-1)} \rightarrow \caO(V)$ invariant
  under Hamiltonian flow is one-dimensional, and spanned by the
  multiplication map. The same holds for polydifferential operators on
  the completion $\widehat{\caO(V)}=\caO(\widehat V)$ of $\caO(V)$
  with respect to the augmentation ideal.
\end{lemma}

\begin{proof}
  It suffices to pass to the formal completion and consider
  polydifferential operators on $\widehat{\caO(V)}$.  Such
  polydifferential operators are determined by their value on elements
  $f^{\otimes (n-1)}$ for $f \in \widehat{\caO(V)}$, since they are
  symmetric and hence determined by their restriction to
  $\Sym^{n-1} \widehat{\caO(V)}$. Furthermore, we can assume that $f'(0)
  \neq 0$, since the complement of this locus in the pro-vector space
  $\widehat{\caO(V)}$ has codimension equal to $\dim V \geq 2$.

  Write $V \cong \bC^{2n}$ with the standard symplectic form
  $\omega = \sum_{i=1}^{\dim V} dx_i \wedge dy_i$.  Applying the
  formal Darboux theorem, there is a formal symplectomorphism of $V$ whose
  pullback takes $f$ to $x_1$ (i.e., $f$ can be completed to a
  coordinate system in which the symplectic form is the standard one),
  so we can assume $f=x_1$. Since all formal symplectic automorphisms
  are obtained by integrating Hamiltonian vector fields, it suffices
  to consider the value $\psi(x_1^{\otimes (n-1)})$.  This value must
  be a function that, in coordinates, depends only on $x_1$, since
  such functions are the only ones which are invariant under all
  symplectic automorphisms fixing $x_1$. By linearity and invariance
  under conjugation by rescaling $x_1$ (and applying the inverse
  scaling to $y_1$), we deduce that
  $\psi(x_1^{\otimes (n-1)}) = \lambda \cdot x_1^{n-1}$ for some
  $\lambda \in \bC$.  Thus, on $x_1^{\otimes (n-1)}$, $\psi$ coincides
  with $\lambda$ times the multiplication operator,
  $f_1 \otimes \cdots \otimes f_{n-1} \mapsto \lambda f_1 \cdots
  f_{n-1}$.
  The latter operator is evidently symmetric and invariant under
  Hamiltonian flow. On the other hand, we have argued that a symmetric
  operator invariant under Hamiltonian flow is uniquely determined by
  its value on $x_1^{\otimes (n-1)}$.  So $\psi$ is equal to $\lambda$
  times the multiplication operator, as desired.
\end{proof}

As a by-product, we obtain the following theorem:

  \begin{theorem}\label{snp1thm} (\cite{hp0weyl}, Theorem 1.6) Let $V$ be a finite dimensional symplectic vector space, and realize $V^{n-1}$ as 
  the set of elements $(v_1 \ldots ,v_n)\in V^n$ such that $\sum_{i=1}^{n}v_i=0$. Then 
  $\HP_0(\caO(V^{n-1} / S_n))^* \cong \bC$, spanned by the augmentation map $\caO(V^{n-1}) \to \bC$.
In other words, the Lie algebra $\mfg$ of $S_n$-invariant Hamiltonian vector fields on $V^{n-1}$ 
is perfect: $\mfg=[\mfg,\mfg]$. 
\end{theorem}

\section{Structure of $M(X)$ for complete intersections with isolated
  singularities}

In this section we discuss results from \cite[\S 5]{ES-dmlv} and
\cite{ES-ciis} concerning $M(X)$ and $\HP^{DR}_*(X)$ when $X$ is a
complete intersection surface with isolated singularities (or more
generally of arbitrary dimension, if one suitably defines the Lie
algebra $\mfg$ of Hamiltonian vector fields). In particular we will
recover topological information about the singularities, including the
Milnor numbers and genera, and find examples where $M(X)$ is not
semisimple.  For concreteness, we will take $X$ to be a surface
in $\bC^3$ throughout most of the section, and explain at the end how
the arguments extend to general (locally) complete intersections (of
arbitrary dimension) in complex affine space.

\subsection{Main results for surfaces in $\bC^3$} Let $X = \{f=0\} \subseteq \bC^3=\Spec \bC[x_1,x_2,x_3]$ be a surface with $f$ irreducible. It is naturally Poisson, with bracket
$\{x_1,x_2\} = f_{x_3}$ together with cyclic
permutations of the indices.  We will be interested particularly in
$f$ having two nice properties:
\begin{definition}
  A variety $X$ is said to have isolated singularities if the singular
  locus of $X$ is finite. A function $f \in \bC[x_1,\ldots,x_n]$ is
  said to have isolated singularities if $\{f=0\}$ is a variety with isolated
  singularities.
\end{definition}
\begin{definition} A function $f \in \bC[x_1,\ldots,x_n]$ is
  quasi-homogeneous of weight $|f|=m \geq 1$ with respect to weights
  $|x_i|=a_i \geq 1$ if $f$ is a linear combination of monomials of
  degree $m$ with respect to these weights.
\end{definition}
Note that $f$ is quasi-homogeneous if and only if $X=\{f=0\}$ is conical
with respect to the $\bC^\times$ action
$\lambda \cdot (x_1,\ldots,x_n) = (\lambda^{a_1} x_1, \ldots,
\lambda^{a_n} x_n)$,
i.e., this action (which contracts to the origin) preserves $X$.
Note that, if $f$ is quasi-homogeneous with isolated singularities,
then since the singular locus is preserved under the action of
$\bC^\times$, it must be identically $\{0\}$ (or empty).
\begin{example}\label{ex:kl-ell}
  In the case that $f \in \bC[x_1,x_2,x_3]$ is quasi-homogeneous with
  weight $m$ with respect to $|x_i|=a_i$, we see that the Poisson
  bracket has degree $d:=m-(a_1+a_2+a_3)$.  If moreover $f$ has
  isolated singularities (i.e., its singular locus is $\{0\}$), then
  $X$ is well known to be isomorphic to
a du Val singularity, $\bC^2/\Gamma$ for 
$\Gamma < \SL(2,\bC)$ a finite subgroup, if $d < 0$; to have a simple
elliptic singularity at the origin if $d=0$; and to have neither a du Val
nor elliptic singularity if $d > 0$.  In particular, up to
isomorphism, in the case $d < 0$ (a du Val singularity) $X$ is
isomorphic to one of the following (see, e.g., \cite{Brisessag} and
\cite[Proposition 2.3.2]{EGdelpezzo}):
\begin{gather} \label{amdesc}
A_{m-1}: \Gamma = \bZ/m, a_1=2, a_2=a_3=m, f = x_1^m + x_2^2 + x_3^2, \\
D_{m+2}: \Gamma = \widetilde{D_{2m}}, a_1=2, a_2=m, a_3=m+1, f = x_1^{m+1} + x_1 x_2^2 + x_3^2, \\
E_6: \Gamma = \widetilde{A_4}, a_1=3, a_2=4, a_3=6,   f = x_1^4 + x_2^3 + x_3^2, \\
E_7: \Gamma = \widetilde{S_4}, a_1=4, a_2=6, a_3=9,  f = x_1^3 x_2 + x_2^3 + x_3^2, \\
E_8: \Gamma = \widetilde{A_5}, a_1=6, a_2=10, a_3=15, f = x_1^5 + x_2^3 + x_3^2; \label{e8desc}
\end{gather}
and in the case $d=0$ (elliptic), then $X$ is isomorphic to one of the following forms:
\begin{gather}
\widetilde {E_6}: a_1=a_2=a_3=1, f= x_1^3 + x_2^3 + x_3^3 + \lambda x_1x_2x_3, \\
\widetilde {E_7}: a_1=a_2=1, a_3=2, f = x_1^4 + x_2^4 + x_3^2 + \lambda x_1x_2x_3,\\
\widetilde {E_8}: a_1=1, a_2=2, a_3=3, f = x_1^6 + x_2^3 + x_3^2 + \lambda x_1x_2x_3.
\end{gather}
\end{example}
\begin{notation} For a possibly singular complex algebraic or analytic
  variety $X$, let $H^*_{\tpl}(X)$ denote the topological cohomology
  of $X$. 
\end{notation}
\begin{notation} For $s \in X$ an isolated singularity, let $\mu_s$ denote the Milnor number at $s$. Let $g_s$ denote the ``reduced genus'' of the singularity at $s$, which we define as
$g_s = \dim H^{\dim X-1}(Y, \caO_Y)$, for $\rho: \widetilde X \to X$ any resolution of singularities and $Y = \rho^{-1}(s)$ (this definition does not depend on the choice of resolution).
\end{notation}
By \cite{Mil-spch, Ham-ltekr}, 
if
$X_t$ is a smoothing of $X$, then for $B(s)$ a small ball about $s$,
$X_t \cap B(s)$ is homotopic to a bouquet of $\mu_s$ spheres of
dimension $\dim X$.  In the case $X$ is a hypersurface in $\bC^n$, the
Milnor number can be described as the codimension in $\caO(X)$ of the
ideal generated by the partial derivatives of $f$.
\begin{theorem} \label{t:hpdr-his} \cite[Theorem 2.4]{ES-ciis}
  If $X$ is a surface in $\bC^3$ with isolated singularities
  $s_1,\ldots,s_k$, then
  $\HP^{DR}_*(X) \cong H_{\tpl}^{2-*}(X) \oplus
  \bigoplus_{i=1}^k\bC^{\mu_{s_i}}$,
  placing $\bC^{\mu_{s_i}}$ in degree zero.
\end{theorem}
Theorem \ref{t:hpdr-his} follows from the following structure theorem for
$M(X)$ in this case. Let $Z \subseteq X$ be the singular locus and
$j: X\setminus Z \to X$ the open embedding.  Recall that, for $M$ a
holonomic right D-module on $X\setminus Z$, there are pushforward
complexes of D-modules $j_! M, j_* M$ on $X$ and a holonomic
 D-module $j_{!*} M$ on $X$ satisfying the adjunction properties
$\Hom(H^0j_! M, N) = \Hom(M, j^! N)$ and
$\Hom(N,H^0j_* M) = \Hom(j^! N,M)$ (for $j^! N = j^* N$ the restriction of $N$ to the open subset $U$); also $j_{!*} M$ is the minimal
extension, which means for example that $j_{!*} M$ is simple if and
only if $M$ is.  Since $X \setminus Z$ is a smooth symplectic variety,
$j^! M(X) = \Omega_{X \setminus Z}$, which implies that $M(X)$ has a
composition series one of whose composition factors is the
intersection cohomology D-module of $X$,
$\IC(X) = j_{!*} \Omega_{X \setminus Z}$, 
and the others which are  D-modules supported on the
singular locus, and hence are delta-function D-modules since the
singular locus is finite. Let $M(X)_{\text{ind}}$ be the
indecomposable summand of $M(X)$ which has $\IC(X)$ as a composition
factor.

\begin{theorem}\label{t:mx-his} \cite[Theorem 2.7]{ES-ciis}
Let $X = \{f=0\} \subseteq \bC^3$ have singular locus $Z = \{s_1, \ldots, s_k\}$ (with the $s_i$ distinct). Then
there is a short exact sequence 
\[
0 \to H^0j_! \Omega_{X\setminus Z} \to M(X) \to \bigoplus_{i=1}^k \delta^{\mu_{s_i}} \to 0.
\]
\end{theorem}
\begin{conjecture}\label{c:mx-his2} \cite[Conjecture 3.8]{ES-ciis}\footnote{For a proof that the conjecture above is the same as the statement from \cite{ES-ciis}, see 
Proposition 3.13
    of the arXiv version 1401.5042 of \cite{ES-ciis}, and apply
    Grothendieck--Serre duality to identify $H^{\dim Y}(Y,\caO_Y)$
    with global sections of the logarithmic canonical bundle on $Y$. See
also the second paragraph of the proof of \cite[Lemma 2.5]{BS-bfun}.}
  Under the same hypotheses as in Theorem \ref{t:mx-his}, we have a
  short exact sequence
\[
0 \to H^0j_! \Omega_{X \setminus Z} \to M(X)_{\text{ind}} \to \bigoplus_i \delta^{g_{s_i}} \to 0.
\]
\end{conjecture}
\begin{theorem}\label{t:mx-his2} \cite[Proposition 3.11]{ES-ciis} Conjecture \ref{c:mx-his2} is true if
  $f$ is quasi-homogeneous. 
\end{theorem}
\begin{remark} \label{r:mx-nonss} Theorem \ref{t:mx-his} shows that
  there are delta-function submodules of $M(X)$ which are not direct
  summands whenever $H^0 j_! \Omega_{X\setminus Z} \ncong \IC(X)$, the
  minimal extension of $\Omega_{X\setminus Z}$. Equivalently, this
  holds whenever $\Ext^1(\IC(X),\delta_s) \neq 0$ for some point
  $s \in Z$.  Since
  $\Ext^1(\IC(X),\delta_s) \cong \Ext^1(\delta_s, \IC(X))^*$ by
  Verdier duality, we actually see that $M(X)$ is not semisimple
  whenever possible, i.e., whenever there exists a non-semisimple
  extension of $\Omega_{X\setminus Z}$ to $X$.  One can compute that
  $\Ext^1(\IC(X),\delta_s) \cong H_{\tpl}^1(U_s \setminus \{s\})$
  where $U_s$ is some contractible neighborhood of $s$ (see \cite[Lemma 3.5]{BS-bfun}, similar to
  \cite[Lemma 4.3]{ESdm}), which exists by \cite{Gil-eavlc}, cf.~also
  \cite[2.10]{Mil-spch}. Here $U_s \setminus\{s\}$ is homotopic to the
  topological \emph{link} of $s$.  Thus we see that $M(X)$ is not
  semisimple whenever the link has nonzero first Betti number.
\end{remark}
\begin{remark} \label{r:mx-nonss-cone} Taking the special case where
  $X$ is a cone over a smooth curve $\Sigma$ in $\bP^2$ of degree $d$,
  recall that the reduced genus of the singularity is
  $g = \frac{(d-1)(d-2)}{2}$ (the same as the genus of $\Sigma$) and the Milnor number is $\mu=(d-1)^3$. In this
case, we see that $H_{\tpl}^1(X \setminus \{0\}) \cong H_{\tpl}^1(\Sigma)$ (since
  $X\setminus \{0\}$ is homotopic to a nontrivial $S^1$-bundle over
  $\Sigma$). Therefore by Remark \ref{r:mx-nonss} we obtain that
  $H^0 j_! \Omega_{X \setminus \{0\}}$ is an extension given by a
  short exact sequence of the form
\begin{equation}
0 \to \delta_0^{2g} \to H^0 j_! \Omega_{X \setminus \{0\}} \to \IC(X) \to 0.
\end{equation}
We obtain from Theorem \ref{t:mx-his2} that
$M(X)_{\text{ind}}$ has a filtration with subquotients
$\delta^{2g}, \IC(X)$, and $\delta^g$, in that order. In particular,
for $g > 0$, it is not semisimple, and moreover not self-dual (there is
twice the multiplicity of $\delta$ function D-modules on the
bottom as on the top). From Theorem \ref{t:mx-his} we see that $M(X) \cong M(X)_{\text{ind}} \oplus \delta_0^{\mu-g}$.
\end{remark}
\subsection{The family $M(X_t)$ of fiberwise D-modules on a smoothing}
The results can be reinterpreted in terms of a smoothing.  Namely, we
can consider all the level sets $X_t := \{f=t\} \subseteq \bC^3$ of
$f$ in $\bC^3$ as $t \in \bC$ varies, which for generic $t$ will be
smooth. It is then well known that, for generic $t$,
$\dim H_{\tpl}^2(X_t) = \dim H_{\tpl}^2(X) + \sum_{s \in Z} \mu_s$, where again
$Z \subseteq X$ is the singular locus which we assume is finite; this
is a consequence of the fact from \cite{Mil-spch, Ham-ltekr} that, for
every $s \in Z$ and $0 < |t| \ll 1$, the intersection of $X_t$ with a
small ball about $s \in \bC^3$ is homotopic to a bouquet of
$\mu_s$ $2$-spheres.  As a result, Theorem \ref{t:hpdr-his} implies
that $\dim \HP_0(\caO(X)) = \dim H_{\tpl}^2(X_t)$ for generic $t$.  Since
$X_t$ is a smooth symplectic surface,
$H_{\tpl}^2(X_t) \cong \HP_0(\caO(X_t))$ (by Example \ref{ex:poisson}). Thus
near $t=0$ the family $\HP_0(X_t)$ of vector spaces has constant
dimension, i.e., forms a vector bundle. 
We have proved:
\begin{corollary} \label{c:his-smooth} \cite[Corollary 1.4]{ES-ciis} Assume
  $X = \{f=0\} \subseteq \bC^3$ has isolated singularities. Then, the
  sheaf $\HP_0(\caO(X_t))$ on $\bC$ is a
  vector bundle near $t=0$ of rank
  $\dim H_{\tpl}^2(X) + \sum_{s \in Z} \mu_s = \dim \HP_0(\caO(X))$.  The
  generic fiber is $H_{\tpl}^2(X_t)$.
\end{corollary}
The entire Poisson-de Rham homology similarly forms a graded vector
bundle \cite[Corollary 2.5]{ES-ciis}, since the dimension of
$H_{\tpl}^i(X_t)$ does not depend on $t$ for $i \in \{0,1\}$, and this equals
the dimension of $\HP^{DR}_{2-i}(X)$.

In fact, this is a shadow of a similar phenomenon for the D-modules
$M(X_t)$, which is central to the proof of the main results.  Note
that, for each $t \in \bC$, we can consider the D-module $M(X_t)$ on
$X_t$, which for $i_t: X_t \to \bC^3$ the embedding, is realized as
$(i_t)_\natural M(X_t)$, a right $\caD(\bC^3)$-module on $\bC^3$
supported on $X_t$.
\begin{theorem}\label{t:mxt}\cite[Theorem 2.8]{ES-ciis}
  The family on $\bC$ of right $\caD(\bC^3)$-modules with fiber
  $(i_t)_\natural M(X_t)$ at $t \in \bC$ is flat near $t=0$.  For generic
  $t$, the fiber of this family is the right D-module is isomorphic to
  $(i_t)_\natural \Omega_{X_t}$.
\end{theorem}
Applying the pushforward functor $\pi_*$, for $\pi: \bC^3 \to \pt$ the
projection to a point, Theorem \ref{t:mxt} recovers
Corollary \ref{c:his-smooth}.
\subsection{Sketch of proof of Theorems \ref{t:hpdr-his}, \ref{t:mx-his}, and \ref{t:mxt}}
Theorem \ref{t:hpdr-his} follows immediately from Theorem
\ref{t:mx-his}. For the latter, since $M(X)$ is locally defined, we
can restrict to a neighborhood $U_s$ of each singular point $s \in Z$,
and if we work in the analytic category, we can take the neighborhood
to be contractible (\cite{Gil-eavlc}, cf.~also \cite[2.10]{Mil-spch}).
Moreover, in this case we can show \cite[Corollary 5.9]{ES-dmlv}
(using results of G.-M.~Greuel \cite{Gre-GMZ}) that the maximal quotient of
$M(U_s)$ supported at the singularity is $\delta^{\mu_s}$, or
equivalently \cite[Theorem 5.11]{ES-dmlv} that
$\HP_0(\caO(U_s)) \cong \bC^{\mu_s}$. The main idea is to compare
$\HP_0(\caO(U_s))$ directly with the cohomology of the de Rham complex
of $U_s$ modulo those differential forms which are torsion (i.e.,
restrict to zero on $U_s \setminus \{s\}$).  This takes care of the
last term of the sequence in Theorem \ref{t:mx-his}.

Theorem \ref{t:mx-his} then reduces to the statement that the
canonical map $H^0j_! \Omega_{X\setminus Z} \to M(X)$ (obtained by
adjunction from
$\Omega_{X \setminus Z} \cong j^!M(X) = M(X)|_{X \setminus Z}$) is
injective. We then prove this and Theorem \ref{t:mxt} simultaneously.
The argument is as follows: in general, if $M(X_t)$ is not
torsion-free, then for each $t$ we can consider the torsion at $t$,
$M(X_t)_{\tor}$, such that $M(X_t)' := M(X_t)/M(X_t)_{\tor}$ is
torsion-free. Note that $M(X_t)_{\tor}$ is zero for generic
$t$. Applying the pushforward functor $\pi_*$ for $\pi: \bC^3 \to \pt$
the projection to a point, we can see that $\pi_* M(X_t)'$ is
quasi-isomorphic to a complex of finitely-generated and torsion-free
$\bC[t]$-modules, i.e., of free finitely-generated
$\bC[t]$-modules. Taking Euler characteristic, we conclude that
$\chi(\pi_* M(X_t)')$ is independent of $t$.  Since $M(X_t)'=M(X_t)$
for generic $t$, we conclude that
$\chi(\pi_* M(X)) = \chi(\pi_* M(X_t)) + \chi(\pi_* M(X)_{\tor})$.
Furthermore, $M(X)_{\tor}$ is supported on the singular locus, as
$M(X)|_{X\setminus Z} = \Omega_{X\setminus Z}$ is simple. Therefore
$M(X)_{\tor}$ is the direct sum of delta submodules, so
$\pi_* M(X)_{\tor}$ is a vector space concentrated in degree zero. We
conclude that
\begin{equation}\label{e:ciis-pf1}
\chi(\HP^{DR}_*(X)) \geq \chi(H_{\tpl}^{2-*}(X_t)) = \chi(H_{\tpl}^{2-*}(X))+\sum_{s
  \in Z} \mu_s.
\end{equation}
Next, $M(X)$ has finite length (as it is holonomic) and its
composition factors consist of a single copy of
$\IC(X) := j_{!*} \Omega_{X\setminus Z}$, for $j: X\setminus Z \to X$
the open embedding, together with some delta-function D-modules
supported on $Z$.  Therefore we also obtain, for
$\IH^*(X) := \pi_* \IC(X)$ the intersection cohomology of $X$,
\begin{equation}\label{e:ciis-pf2}
\chi(\HP^{DR}_*(X)) = \chi(\IH^*(X)) + m,
\end{equation}
where $m$ is the number of delta-function D-modules appearing as composition factors of $M(X)$. Putting \eqref{e:ciis-pf1} and \eqref{e:ciis-pf2} together, we get
\begin{equation}
\chi(\IH^*(X)) + m \geq \chi(H_{\tpl}^{2-*}(X))+\sum_{s
  \in Z} \mu_s.
\end{equation}
Now we restrict again to $X=U_s$, a contractible neighborhood of $s$.
As we saw above, in this case $\mu_s$ is the length of the maximal
quotient of $M(U_s)$ supported at $s$. Let $K$ be the 
maximal submodule of $M(X)_{\text{ind}}$ supported at $s$. Let
$m'$ be the length of $K$ (so $K \cong \delta_s^{m'}$).  Then we have
an exact sequence $0 \to K \to M(X)|_{U_s} \to \delta^{\mu_s} \to 0$, and
hence $m=\mu_s+m'$.
Since $U_s$ is contractible, we obtain from all of this:
\begin{equation}\label{e:ciis-pf3}
\chi(\IH^*(U_s)) + m' \geq 1.
\end{equation}
Finally we explicitly compute $\chi(\IH^*(U_s))$ \cite[Proposition
5.9]{ES-dmlv} using the classical formula \cite[\S 6.1]{GM-iht},
obtaining $\chi(\IH^*(U_s)) = 1 - \dim H_{\tpl}^1(U_s \setminus \{s\})$. We
conclude $m' \geq \dim H_{\tpl}^1(U_s \setminus \{s\})$.  On the other hand,
we compute in general that
$\dim H_{\tpl}^1(U_s \setminus \{s\}) = \dim \Ext^1(\IC(U_s), \delta_s)$, the
maximum possible size of a submodule of $M(X)$ supported at $s$ which does not
contain a direct summand of $M(X)$.  In other words, this is one less than the
length of $H^0j_! \Omega_{X\setminus Z}$ itself. Therefore the map
$H^0 j_! \Omega_{X\setminus Z} \to M(X)$ is injective, and moreover
\eqref{e:ciis-pf3} is an equality.  The former finishes the proof of
Theorem \ref{t:mx-his}, and the latter proves Theorem \ref{t:mxt}.
 
\subsection{Sketch of proof of Theorem \ref{t:mx-his2}}\label{ss:mx-his2-pf}
Here, following \cite[\S 6]{ES-ciis}, we sketch the proof of Theorem
\ref{t:mx-his2} under the additional assumption that $f$ is actually
homogeneous, i.e., $a_i=|x_i|=1$ for all $i$.  The general case does
not significantly change the details, and this makes things a bit more
concrete and easier to follow. 

Let $\Eu = \sum_i x_i \partial_i$ be the Euler vector field.  The main
idea is that there is a canonical endomorphism
$T_{\Eu} \in \End_{\caD(\bC^3)}(i_\natural M(X))$ which descends from the
endomorphism $\Phi \mapsto \Eu \cdot \Phi$ of left multiplication on
$\caD(\bC^3)$.  (Actually, $M(X)$ is weakly equivariant, which means
that $i_\natural M(X)$ is a graded right $D(\bC^3)$-module, such that for
homogeneous $\Phi \in i_\natural M(X)$,
$|\Phi| \Phi = T_{\Eu}(\Phi)-\Phi \cdot \Eu$; see Section
\ref{s:whc}.)  Let $T_{\Eu}$ also denote the corresponding endomorphism
of $M(X)$. Then $M(X)$ decomposes
as a direct sum of the generalized eigenspaces of $T_{\Eu}$; call them
$M(X)_m$.

Next, the degree in which $M(X)_{\text{ind}}$ appears is the degree
$d$ of the Poisson bivector $\sigma$, since the action of ${\Eu}$ by
right multiplication on the symplectic volume form of
$X\setminus\{0\}$ is multiplication by $d$, and the symplectic volume
form is the image of the generator $1$ under the canonical quotient
$\caD_{X\setminus\{0\}} \onto \Omega_{X\setminus\{0\}}$ (as it is
invariant under Hamiltonian flow).  So $M(X)_d$ is a direct sum of
$M(X)_{\text{ind}}$ and some $\delta$-function D-modules.  By
Theorem \ref{t:mx-his}, we have an exact sequence
\begin{equation}\label{e:mx-his2-es-pf}
0 \to H^0 j_! \Omega_{X \setminus \{0\}} \to M(X)_d \to \delta_0 \otimes \HP_0(\caO(X))_d \to 0.
\end{equation}
Since the Poisson bracket has degree $d$ and the only elements of
weight zero are constants which are central, $\{\caO(X),\caO(X)\}$ is
spanned by elements of weight $> d$, and hence
$\HP_0(\caO(X))_d = \caO(X)_d$.  In \cite[Proposition 3.11]{ES-ciis}
this space is shown to have dimension $g_0$ (using the identification
of $H^{\dim X - 1}(D,\caO_D)$ with the vector space appearing in
\cite[Conjecture 3.8]{ES-ciis} as we remarked in the footnote to
Conjecture \ref{c:mx-his2}).
%
 Therefore Theorem \ref{t:mx-his2} reduces to the statement that $M(X)_d$
itself is indecomposable, hence $M(X)_d = M(X)_{\text{ind}}$.

To show this, we first claim that for each $m \in \bZ$, $T_{\Eu}$ is actually
central in $\End(M(X)_m)$. Indeed, every endomorphism of $M(X)$
descends from left multiplication
$T_{\Psi}: \Phi \mapsto \Psi \cdot \Phi$ by some $\Psi \in
\caD(\bC^3)$.
Since $[T_{\Eu},T_{\Psi}] = T_{[\Eu,\Psi]}$ and $[\Eu,-]$ is a semisimple
operator on $\caD(\bC^3)$, it follows that $[T_{\Eu},-]$ is a semisimple
operator on $\End(M(X))$.  On each generalized eigenspace $M(X)_m$ we
therefore have that $[T_{\Eu},\End(M(X)_m)]=0$.  

As a result, $T_{\Eu}$ preserves every direct summand of $M(X)_m$ for
every $m$. For a contradiction, suppose that
there is a direct summand $K$ of $M(X)_d$ supported at the origin. We can assume
$K \cong \delta_0$.  Now  $T':=T_{\Eu}-d\Id$ is a nilpotent
endomorphism of $M(X)_d$ which preserves
$K$, hence $T'|_K=0$. 

To obtain a contradiction, let $v: M(X)_d \to K \cong \delta_0$ be the projection. Let $N$ be the space of smooth distributions on $\bC^3$ and $w: \delta_0 \to N$ the inclusion. The main step in the proof is the following (which is a slightly weaker version of \cite[Lemma 6.1]{ES-ciis}):
\begin{lemma} \label{l:mx-his2} There is a Hamiltonian-invariant distribution $\phi_v \in N$ such that $\phi_v \cdot (\Eu-d) = w(v(1))$.
\end{lemma}
Equivalently, there is a solution $\varphi \in \Hom(M(X)_d, N)$, with
$\varphi(1)=\phi_v$, such that $\varphi \circ T' = w \circ v$. Since
$w \circ v|_{K} \neq 0$, it follows that $T'|_K \neq 0$, which is a
contradiction.

To prove Lemma \ref{l:mx-his2}, we construct $\phi_v$ explicitly. By \eqref{e:mx-his2-es-pf} we can associate to $v$ an element of $\HP_0(\caO(X))_d^*$.
Let $Q \in \HP_0(\caO(X))_d$ be the dual of $v$ under the Hermitian pairing on $\HP_0(\caO(X))_d$:
\[
\langle P,Q \rangle := \int_{S^{5} \cap X}P\overline{Q} \omega \wedge \overline{\omega}/dr,
\]
where $\omega$ is the rational volume form on $X$ (the inverse of the Poisson bivector), and $r$ is the radial coordinate function.  Then we partially define $\phi_v$ by:
\[
\phi_v(\alpha)=-2\int_X \alpha \omega \wedge \overline{Q} \overline{\omega}.
\]
This converges when $\alpha$ vanishes to order greater than $d$ \cite[Lemma 6.2]{ES-ciis}, and we can extend it arbitrarily to a linear functional on all of $C_c^{\infty}(\bC^3)$. The result is annihilated by all Hamiltonian vector fields which vanish to order greater than $d$, but since the weight of $\sigma$ is $d$, this includes all Hamiltonian vector fields.  Thus $\phi_v \in \Hom(M(X)_d,N)$.

We claim that $\phi_v \cdot (\Eu-d) = w(v(1))$, which completes the proof. 
By definition of $Q \in \HP_0(\caO(X))_d = \caO(X)_d$, the
RHS has the property 
$w(v(1))(PH) =  \langle Q,P\rangle H(0)$ for all $P \in \caO(X)_d$ and $H \in C_c^\infty(\bC^n)$.  We have to show the same for the LHS.  It suffices to show this for a single choice of $H$, and the identity therefore reduces to an explicit computation: choosing  $H(x)=h(|x|^2)$ for $h$ a smooth real function with $h(0) \neq 0$, we obtain
\[
(\phi_v \cdot (\Eu-d))(PH)=-2\int_X P\overline{Q} \Eu(H) \omega \wedge \overline{\omega} = -2\int_0^\infty dr \cdot r \cdot h'(r^2) \cdot \langle P,Q \rangle = \langle P,Q \rangle h(0) = \langle P,Q \rangle H(0).
\]
The first equality uses that we can choose the extension of $\phi_v$
from functions vanishing to order $> d$ at $0$ to all functions in
such a way that we project away from a space of functions supported in
an arbitrarily small neighborhood of zero (and that
$\phi_v \cdot (\Eu-d)$ does not depend on the choice of extension).  This completes the sketch.


\subsection{Generalization to locally complete intersections and higher dimension}\label{ss:gen-lci-hd}
Finally, we explain how to generalize from surfaces in $\bC^3$ to general
locally complete intersections of arbitrary dimension.

Suppose that $X$ is a complete intersection surface in $\bC^n$, i.e.,
$X = \{f_1=\cdots=f_{n-2}=0\}$.  Then $X$ has a Poisson
structure, called the Jacobian Poisson structure, given by the
skew-symmetric biderivation
$\sigma = i_{\partial_1 \wedge \cdots \wedge \partial_n} (df_1 \wedge
\cdots \wedge df_{n-2})|_X$,
i.e.,
$\sigma = \sum_{i<j} \sigma_{i,j} (\partial_i \wedge \partial_j)|_{X}$
for
$\sigma_{i,j} =
(-1)^{i+j-1}\frac{\partial(f_1,\ldots,f_{n-2})}{\partial(x_1,\ldots,\widehat{x_i},\ldots,\widehat{x_j},\ldots,x_n)}$
the Jacobian determinant (omitting $x_i$ and $x_j$ from the
denominator). Explicitly, 
\[
\{g,h\} = \frac{\partial(f_1,\ldots,f_{n-2},g,h)}{\partial(x_1,\ldots,x_n)}.
\]

With this modification all of the preceding results generalize without
change, except that now the family $X_t$ should be defined by picking
a generic line $\bC \cdot (c_1,\ldots,c_{n-2}) \subseteq \bC^{n-2}$,
and setting $X_t = \{f_i=c_i\}_{i=1}^n \subseteq \bC^n$, which will be
smooth for generic $t$. The quasi-homogeneous condition is replaced by
the condition that \emph{all} of the $f_i$ be quasi-homogeneous with
respect to fixed weights $a_i = |x_i|$ (but the degrees of the $f_i$
can be different).

Next, we explain how to generalize to varieties of higher
dimension. If $X=\{f_1=\cdots=f_{n-k}=0\} \subseteq \bC^n$ is a
complete intersection of dimension $k \geq 2$, then $X$ is equipped with a
canonical skew-symmetric multiderivation
$\Xi_X: \caO(X)^{\otimes \dim X} \to \caO(X)$,
\[
\Xi_X := i_{\partial_1 \wedge \cdots \wedge \partial_n}(df_1 \wedge
\cdots \wedge df_{n-k}) |_{X}.
\]

Given a $(k-2)$-form $\alpha \in \Omega^{k-2}(X)$, we can associate a
vector field, $\xi_\alpha:=i_{\Xi_X}(d\alpha)$, which we call the
\emph{Hamiltonian vector field of $\alpha$}.  Let $\mfg$ be the set of
all Hamiltonian vector fields on $X$. We can check that $\mfg$ is a
Lie algebra, using the formula
$[\xi_\alpha, \xi_\beta] = \xi_{L_{\xi_\alpha} \beta}$ \cite[\S
3.4]{ES-dmlv}, which one can check in local coordinates.
\begin{remark}\label{r:ls-same-g}
  The vector fields $\xi_\alpha$ actually make sense as vector fields
  on all of $\bC^n$: since
  $\widetilde{\alpha} := df_1 \wedge \cdots \wedge df_{n-k} \wedge
  \alpha$
  is a well-defined $(n-1)$-form on $\bC^n$, we obtain
  $\widetilde{\xi_\alpha} := i_{\partial_1 \wedge \cdots
    \wedge \partial_n}(df_1 \wedge \cdots \wedge df_{n-k} \wedge
  \alpha)$,
  a well-defined vector field on $\bC^n$, which is Hamiltonian on each
  level set $X_{c_1,\ldots,c_{n-k}} = \{f_i=c_i\}_{i=1}^{n-k}$.  Thus
 the Lie algebra of Hamiltonian vector fields on each
  $X_{c_1,\ldots,c_{n-k}}$ is the same Lie algebra $\mfg$.
\end{remark}

Thus, we can replace $\HP_0(\caO(X))$ by
$(\caO(X))_{\mfg}=\HP_0^{\mfg-DR}(X)$ and similarly for $X_t$ (where
now we define $X_t$ using a generic line
$\bC \cdot (c_1,\ldots,c_{n-k}) \subseteq \bC^{n-k}$).  Similarly we
can replace $M(X)$ and $M(X_t)$ by $M(X,\mfg)$ and $M(X_t,\mfg)$
(in view of Remark \ref{r:ls-same-g}, we have the explicit description
$(i_t)_\natural M(X_t,\mfg)=(\mfg \caD_{\bC^n} + I_{X_t}
\caD_{\bC^n})\setminus \caD_{\bC^n}$.)
With this change, all of the results and proofs continue to go through
in the same manner.

Finally, since the D-module $M(X)$ is defined locally, we actually
don't need to require that $X$ globally be a complete intersection,
only that it locally admits this structure, provided that there is a
global skew-symmetric multiderivation
$\Xi_X: \caO_X^{\otimes \dim X} \to \caO_X$ which is nonvanishing on
the smooth locus. We do not require $X$ to be affine.  The notion of
isolated singularities still makes sense. We can replace the
quasi-homogeneous condition by the condition that, at every point
$x \in X$, the formal neighborhood $\hat X_x$ of the point is
isomorphic to a complete intersection in $\hat \bC^n_0$ defined by
quasi-homogeneous functions (with respect fixed weights that can
depend on the point $x$ chosen), i.e., that $X$ be formally locally
conical.

\subsection{Relationship to the Bernstein--Sato polynomial}\label{ss:RBSp}
In \cite{BS-bfun}, following a suggestion of the first author, the
second author and T.~Bitoun relate these results to the theory of the
Bernstein--Sato polynomial.  Namely, the latter involves the study, for
$f \in \bC[x_1,\ldots,x_n]=\caO(\bC^n)$ and $\lambda \in \bC$, of the left
$\caD(\bC^n)$-module $\caD(\bC^n) \cdot f^\lambda$. Let us recall
the definition. Let $U := \bC^n \setminus \{f=0\}$, so that
$\caD(U) = \caD(\bC^n)[f^{-1}]$ and $\caO(U) = \caO(\bC^n)[f^{-1}]$. Let
$\caD(U) \cdot f^\lambda$ be the $\caD(U)$-module which
as a $\caO(U)$-module is the trivial line bundle on $U$ together with a nonvanishing section denoted $f^\lambda$, and the $\caD(U)$ action is defined
by the formula $\xi \cdot f^\lambda = \lambda \xi(f) f^{-1} \cdot f^{\lambda}$ for every derivation $\xi \in \Vect(U)$. In other words,
$\caD(U) \cdot f^\lambda = \caD(\bC^n)[f^{-1}]/\caD(\bC^n)[f^{-1}] \cdot (\partial_i -
\lambda f^{-1} f_{x_i} \mid 1 \leq i \leq n)$ with generator denoted by $f^\lambda$.
Let $\caD(\bC^n) \cdot f^\lambda$ be the $\caD(\bC^n)$-submodule of $\caD(U) \cdot f^\lambda$ generated by $f^\lambda$. We similarly define
the
$\caD(\bC^n)[s]$-module $\caD(\bC^n)[s] \cdot f^s$.
For every right $\caD(\bC^n)$ module, let
$M^{\ell} := M \otimes_{\caO(\bC^n)} \Omega_{\bC^n}^{-1}$ be the
associated left $\caD(\bC^n)$-module. For $X = \{f=0\}$, one has the
following canonical surjection of left $\caD(\bC^n)$-modules:
\begin{equation}\label{e:bfun-map-l}
i_\natural M(X,\mfg)^\ell = i_\natural M(X,\mfg) \otimes_{\caO(\bC^n)} \Omega_{\bC^n}^{-1} \onto  \caD(\bC^n) f^\lambda / \caD(\bC^n) f^{\lambda+1}, \quad 1 \otimes (\partial_1 \wedge \cdots \wedge \partial_n) \mapsto f^\lambda,
\end{equation}
as well as a canonical map
\begin{equation}\label{e:bfun-map-s}
i_\natural M(X,\mfg)^\ell \to \caD(\bC^n)[s] f^s / \caD(\bC^n)[s] f^{s+1},
\end{equation}
which is surjective for $f$ quasi-homogeneous (since then for $\Eu$ the
Euler vector field, we have $|f|^{-1} \Eu \cdot f^s = sf^s$).
Using the maps \eqref{e:bfun-map-l} and \eqref{e:bfun-map-s}, we prove
the following. Let $Z \subseteq X$ continue to be the singular locus,
which we assume to be finite. For each $s \in Z$, let $U_s$ be a contractible neighborhood of $s$. Define
\[
K := \ker(H^0 j_!\Omega_{X\setminus Z} \onto j_{!*} \Omega_{X\setminus
  Z}) \cong \bigoplus_{s \in Z} \Ext^1(j_{!*} \Omega_{X\setminus Z},
\delta_s)^* \otimes \delta_s \cong \bigoplus_{s \in Z} H_{\tpl}^{n -2}
(U_s \setminus \{s\})^* \otimes \delta_s
\]
(for the last isomorphism see \cite[Lemma 3.5]{BS-bfun},
cf.~Example \ref{r:mx-nonss}). Then:
\begin{theorem}\cite{BS-bfun}  Suppose that $n \geq 3$ and
$\{f=0\}$ has isolated singularities. Then 
  \eqref{e:bfun-map-l} produces an isomorphism
  $(i_\natural M(X,\mfg)_{\text{ind}}/K)^\ell \cong
  \caD(\bC^n) f^{-1} / \caO(\bC^n)$.
  If moreover $f$ is quasi-homogeneous, \eqref{e:bfun-map-s} produces an
  isomorphism
  $i_\natural M(X,\mfg)^\ell \cong \caD(\bC^n)[s] f^s / \caD(\bC^n)[s] f^{s+1}$.
\end{theorem}
\begin{corollary} Conjecture \ref{c:mx-his2} is equivalent to the
  statement: if $\{f=0\}$ is irreducible with
isolated singularities, then the length of the left $\caD(\bC^n)$-module
  $\caD(\bC^{n}) \cdot f^{-1} / \caO(\bC^n)$ is one more than the sum of
  the reduced genera of the singularities.
\end{corollary}
This formula was motivated by a similar phenomenon discovered by
T.~Bitoun in characteristic $p > 0$: Let $f$ be a polynomial with
rational coefficients in $n \geq 3$ variables which is absolutely
irreducible, i.e., it remains irreducible over $\bC$, and assume that
$\{f=0\}$
has a unique singular point. For a field $F$ let $\bA_F^n$ be the affine
space over the field $F$ and $\caD(\bA_F^n)$ be the ring of Grothendieck
differential operators on $\bA_F^n$.  Then for $p$ sufficiently large,
we can consider the local cohomology $\caD(\bA_{\bF_p}^n)$-module
$\caO(\bA_{\bF_p}^n)[f^{-1}] / \caO(\bA_{\bF_p}^n)$, which turns out
to be generated by $f^{-1}$ (hence fully analogous to
$\caD(\bC^n) \cdot f^{-1} / \caO(\bC^n)$ in characteristic zero).  In
\cite{bitoun2017length}, it is proved that, for sufficiently large $p$
the length of this $\caD(\bA_{\bF_p}^n)$-module is one more than the ``$p$-genus''
of the singularity, which is defined as the dimension of the ``stable
part'' of $H^{n-2}(Y_{\overline p}, \caO_{Y_{\overline p}})$, for
$Y_{\overline p}$ the exceptional fiber of a resolution of
singularities of
$X_{\overline p}:=\{f=0\} \subseteq \bA_{\overline{\bF_p}}$, where the
stable part refers to the intersection of the images of all powers of the Frobenius
morphism.  See \cite{bitoun2017length,BS-bfun} for details.

The above results also have a direct application to the Bernstein--Sato
polynomial.  Recall that the Bernstein--Sato polynomial is defined as
the annihilator in $\bC[s]$ of
$\caD(\bC^n)[s] f^s / \caD(\bC^n)[s] f^{s+1}$. Using the surjections
$\caD(\bC^n)[s] f^s \onto \caD(\bC^n) \cdot f^\lambda$, it follows
that, if
$\caD(\bC^n) \cdot f^\lambda \neq \caD(\bC^n) \cdot f^{\lambda+1}$,
then $\lambda$ is a root of this polynomial. The converse question is
interesting and little is known aside from \cite[Proposition
6.2]{Kas-BfhsRrBf}, which says that the converse holds if additionally
$\lambda-m$ is not a root for any positive integer $m$. The results
above imply a converse in the quasi-homogeneous case:
\begin{corollary} \cite{Sai-dmgrphf,BS-bfun} For $f$ quasi-homogeneous
  with an isolated singularity and $\lambda \in \bC$, if
  $\lambda$ is a root of the Bernstein--Sato polynomial, then
$\caD(\bC^n) \cdot f^\lambda \neq \caD(\bC^n) \cdot f^{\lambda+1}$.
\end{corollary}
M.~Saito has also demonstrated in \cite[Example 4.2]{Sai-dmgrphf} a
counterexample to this assertion when the quasi-homogeneity assumption
is removed.

\section{Weights on homology of cones}\label{s:whc}

We continue to work over $\bk=\bC$.
In Section \ref{ss:mx-his2-pf}, we used that if $\caO(X)$ is graded
(i.e., $X$ admits a $\bC^\times$-action) and Poisson with a
homogeneous Poisson bracket (which was the case there for $X$ a
quasi-homogeneous surface in $\bC^3$), then $\HP_0(\caO(X))$ is also
nonnegatively graded.  In this section we explain how in fact
$\HP_*^{DR}(X)$ is bigraded, and consider its bigraded
Hilbert-Poincar\'e polynomial (i.e., a Laurent polynomial in two
variables). We give several examples (without proofs) that demonstrate
that one recovers in this way many special polynomials from
combinatorics.

\subsection{Weakly equivariant D-modules and bigrading on Poisson-de Rham homology}
As noted in Section \ref{ss:mx-his2-pf} in a particular situation,
when $\caO(X)$ is graded, then $M(X)$ is weakly
$\bC^\times$-equivariant. Let us first recall the general definition:
\begin{definition} 
Let $\caO(X)$ be graded and $i: X \to V$ a $\bC^\times$-equivariant embedding into a smooth variety $V$, i.e., $\caO(V)$ is graded and $I_X$ is a graded ideal.  Then a weakly $\bC^\times$-equivariant  D-module $M$ on $X$ is one associated
to a graded right $\caD(V)$-module $i_\natural M$ on $V$ supported on $i(X)$. 
\end{definition}
 Call this grading on $i_\natural M(X)$ the \emph{weight grading}. In this case, the pushforward $\pi_* M(X) = i_\natural M(X) \otimes^L_{\caD(V)} \caO(V)$ under the map $\pi: X \to \text{pt}$ is a complex of weight graded vector spaces, since $\caO(V)$ is also a graded $\caD(V)$-module. Therefore the Poisson-de Rham homology $\HP^{DR}_*(X)$ is bigraded.  It is very interesting to compute this grading, which turns out to be closely related to important special polynomials.  
\begin{remark}
  The weight grading on $M(X)$ comes from a decomposition of $M(X)$
  itself. Let $\Eu_X$ and $\Eu_V$ be the Euler vector fields on $X$
  and $V$, i.e., $\Eu_X(f) = |f| f$ when $f \in \caO(X)$ is
  homogeneous and similarly for $V$.  Then the endomorphism
  $\Phi \mapsto \Eu_V \cdot \Phi$ of $\caD(V)$ induces an endomorphism
  of $M(X)$, which can also be described as (i) the endomorphism
  induced by $\Phi \mapsto \Eu_X \Phi$ on the D-module
  $\caD_X$ on $X$, and (ii) the endomorphism of $i_\natural M$ which on homogeneous
  elements is $\Phi \mapsto |\Phi| \Phi - \Phi \cdot \Eu_V$.  Call
  this endomorphism $T_{\Eu}$.  Then $M(X)$ decomposes into its
  generalized eigenspaces under $T_{\Eu}$,
  $M(X) = \bigoplus_{m \in \bZ} M(X)_m$, and $\pi_* M(X)_m$ is
  concentrated in weight $m$. This decomposition is important to the
  proof of the following facts, as well as yielding interesting
  information in its own right (it allows one to define, for instance,
  canonical filtrations on irreducible representations of the Weyl
  group: see Section \ref{ss:nilcone} below). 
\end{remark}

\subsection{Symmetric powers of conical surfaces with isolated singularities}
Let $Y \subseteq \bC^3=\Spec \bC[x_1,x_2,x_3]$ be a conical surface with
singular locus $\{0\}$, i.e., $Y = \{f=0\}$ where $f$ is
quasi-homogeneous with isolated singularities. In this context,
$\bigoplus_{n \geq 0} \HP_0(\caO(S^n Y))^*$ is again a graded algebra
just as in Section \ref{ss:hp0weyl-main}, with grading given by $n$,
which we call the ``symmetric power''. In fact, it is a bigraded
algebra, with the additional grading given by the weight on each
summand $\HP_0(\caO(S^n Y))^*$ (which will be nonpositive, as it is
dual to the nonnegative weight on $\caO(S^n Y)$). It turns out to be
possible to explicitly compute this bigraded vector space:
\begin{theorem}\cite[Theorem 1.1.14]{ESsym} 
There is a (noncanonical) isomorphism of bigraded algebras,
\begin{equation}
\Sym(\HP_0(\caO(Y))^*[t]) \cong \bigoplus_{n \geq 0} \HP_0(\caO(S^n Y))^*,
\end{equation}
assigning $t$ bidegree $(1,-|f|)$ (with $1$ the symmetric power and $-|f|$ the weight), and $\HP_0(\caO(S^n Y))^*$ is assigned symmetric power $n$ (on both sides of the equation).
\end{theorem}
Let $J_Y := \caO(\bC^3) / (f_{x_1},f_{x_2},f_{x_3})$ be the Jacobi ring (so $\dim J_Y = \mu_Y$ is the Milnor number of the singularity at the origin), and let its Hilbert-Poincar\'e polynomial be $h(J_Y;t) = \sum_{i=1}^{\mu_Y} t^{n_i}$.  Note that for $Y = \bC^2/\Gamma$ with $\Gamma < \SL(2,\bC)$ finite, it is well known that $n_i = d_i-2$ for $d_i$ the degrees of the fundamental invariants of the Weyl group attached to $\Gamma$ by the McKay correspondence (see \eqref{amdesc}), i.e., $\bC[\mathfrak{h}]^W$ is a polynomial algebra generated by elements of degrees $d_i$ with $\mathfrak{h}$ the reflection representation of $W$ and $|\mathfrak{h}^*|=1$.


\begin{corollary} \cite[(1.1.16)]{ESsym}
 $\displaystyle \sum_{n \geq 0} h(\HP_0(\caO(S^n Y))^*;t) s^n = \prod_{i=1}^{\mu_Y} \prod_{j \geq 0} \frac{1}{1-t^{n_i + jd}s^{j+1}}$.  In particular the dimension of $\HP_0(\caO(S^n Y))$ is $a_n(\mu_Y)$ \eqref{e:ani-defn}, the number of $\mu_Y$-multipartitions of $n$.
\end{corollary}
\begin{remark}
  We put ``noncanonical'' in parentheses because, unlike in the case
  of \eqref{e:symphp0thm}, we do not construct an explicit
  isomorphism. Note that the map of \eqref{e:symphp0thm} still makes
  sense here, but is no longer injective: indeed, if
  $\varepsilon_n: \caO(S^n Y) \to \bC$ is the augmentation map, then
  the map of \eqref{e:symphp0thm} sends $\varepsilon_1 t^{m-1}$ to
  $\varepsilon_m$ for all $m$, but
  $\varepsilon_{m_1} \cdots \varepsilon_{m_k} =
  \varepsilon_{m_1+\cdots+m_k}$
  in the RHS, which shows that $\varepsilon_1 \& \varepsilon_1$ and
  $\varepsilon_1 t$ both map to $\varepsilon_2$. It does not appear
  that there exists a canonical isomorphism in general.
\end{remark}
Now let us restrict to the special case where $Y \cong \bC^2/\Gamma$ for $\Gamma < \SL(2,\bC)$ finite, i.e., $Y$ is a du Val singularity. Then $S^n Y = \bC^{2n}/(\Gamma^n \rtimes S_n)$ is a finite linear quotient of a vector space, so the special case of \eqref{e:expl-symp-qt-dmod} with $V$ linear (i.e., \cite[Proposition 4.16]{ESdm}), yields the entire structure of $M(S^n Y)$ for every $n$.  We can therefore compute the entire Poisson-de Rham homology, with the help of \cite[Theorem 5.1]{PS-pdrhhvnc}, which tells us in this case that a slight modification of \eqref{e:expl-symp-qt-dmod} is an isomorphism of weakly $\bC^\times$-equivariant D-modules, applying on the right-hand side a shift in weight down by $\dim Z$ in each summand. The result is:
\begin{corollary}
  For $Y=\bC^2/\Gamma$, there is a (noncanonical) isomorphism of
  trigraded algebras,
\begin{equation}
\Sym(\HP_0(\caO(Y))^*[t] \oplus \bC[s] \cdot u) \cong \bigoplus_{n \geq 0} \HP_*^{DR}(\caO(S^n Y))^*,
\end{equation}
where in the trigrading by symmetric power, weight, and homological degree, $|t|=(1,-|f|,0)$, $|s|=(1,0,2)$, and $|u|=(1,-2,2)$. Thus on  trigraded Hilbert series we obtain:
\begin{equation}
\sum_{n,m \geq 0} h(\HP_m^{DR}(\caO(S^n Y))^*;t) s^nu^m = 
\prod_{j \geq 0} \frac{1}{1-s^{j+1} t^{-2}u^{2j}} \prod_{i=1}^{\mu_Y} \prod_{j \geq 0} \frac{1}{1-t^{n_i + jd}s^{j+1}}.
\end{equation}
In particular $\dim \HP^{DR}_*(S^n Y) = a_n(\mu_Y+1)$, the number of $(\mu_Y+1)$-multipartitions of $n$ (where $\mu_Y$ also equals the number of irreducible representations of $\Gamma$).  
\end{corollary}

\subsection{The nilpotent cone}\label{ss:nilcone}
Next let $X \subseteq \mfg$ be the cone of nilpotent elements in a
semisimple Lie algebra $\mfg$. Then as explained in Example
\ref{ex:nilcone},
$\HP_*^{DR}(X) \cong H^{\dim X - *}(T^* \mathcal{B})$, for
$T^* \mathcal{B} \to X$ the Springer resolution. This is isomorphic
to the cohomology of the flag variety $\mathcal{B}$ itself.
However, the source $\HP_*^{DR}(X)$
has an additional grading given from the dilation action on $X$.
G.~Lusztig proposed a formula for the graded Hilbert-Poincar\'e polynomial (\cite[Conjecture 8.1]{PS-pdrhhvnc}).  For $W$ the Weyl group attached to $\mfg$, and $\chi$ an irreducible representation of $W$,   let $K_{\mfg,\chi}(t)$ be the generalized Kostka polynomial,
$K_{\mathfrak{g},\chi}(t) := \sum_{i \geq 0} t^i\dim \Hom_W(\chi,
H^{2\dim\caB - 2i}(\caB))$. Then G.~Lusztig's suggestion was:
\begin{equation}\label{e:lusztig}
h(\HP_*^{DR}(X);x,y)=\sum_{\chi \in \text{Irrep}(W)} K_{\mfg,\chi}(x^2)
  K_{\mfg,\chi}(y^{-2}).
\end{equation}
This is now a theorem \cite[1.1]{BS-KpncSfc}, again making use of \cite{HKihs}.

Note that the fact that $H^{2 \dim \mathcal{B}-*}(\mathcal{B})$ has a
$W$-action and that $\HP^{DR}_*(X)$ has a second grading means that the
isomorphism $\HP^{DR}_*(X) \cong H^{2 \dim \mathcal{B}-*}(\mathcal{B})$ cannot be canonical
(indeed, $K_{\mfg,\chi}(y^{-2})$ would have to be a multiple of
$\dim \chi$ if there were a bigrading compatible with the $W$ action,
and this is false in general).  Instead, in \cite{BS-KpncSfc}, the authors
construct a \emph{canonical family of filtrations} on
$H^{2 \dim \mathcal{B}-*}(\mathcal{B})$, parameterized by
$\lambda \in \mathfrak{h}^*_{\reg}$, with $\mathfrak{h}$ a Cartan
subalgebra of $\mfg$ and $\mathfrak{h}^*_{\reg}$ the complement of the
coroot hyperplanes.
\begin{theorem}\label{t:can-filt}\cite[Theorem 1.3]{BS-KpncSfc}
  For every element $\lambda \in \mfh_\reg^*$, there is a
  canonical associated filtration $\mathcal{F}_\lambda$ on
  $H^{2 \dim \mathcal{B}-*}(\mathcal{B})$ whose associated graded vector space is
  $\HP_*^{DR}(X)$. This is $W$-equivariant:
  $\mathcal{F}_{w(\lambda)} = w(\mathcal{F}_\lambda)$.
\end{theorem}
\begin{corollary}\label{c:irr-filt} \cite[Corollary 1.4]{BS-KpncSfc} 
  On every irreducible representation $\chi$ of $W$,
to every element $\lambda \in \mfh_\reg^*$ is associated
 a canonical filtration whose
  Hilbert series is $K_{\mfg,\chi}(y^{-2})$.
\end{corollary}
\begin{example}\cite[Example 1.5]{BS-KpncSfc}
  Let $\mfg=\mathfrak{sl}_n$ and let
  $\chi = \mfh^* \otimes \text{sign} \cong \bC^{n-1}$.
  Consider $\chi$ to be (in coordinates)
  $\bC^n / \bC \cdot (1,1,\ldots,1)$, and make the same identification
  for $\mfh^*$.  Let $\lambda \in \mfh_{\reg}^*$ be the image of
  $(a_1, \ldots, a_n) \in \bC^n$.  Then the resulting filtration on
  $\chi$ (the Vandermonde filtration) is:
  $F^{2i-2\dim \mathcal{B}}(\chi)$ is the span of $(a_1^j, \ldots, a_n^j)$ for $j \leq i$.
\end{example}
Finally, we can apply this to the classical W-algebras (see Example
\ref{ex:cl-walg}). This uses the
Springer correspondence, which associates to every irreducible
representation $\chi$ of $W$ a pair of a nilpotent coadjoint orbit
$O_\chi \subseteq \mfg^*$ and a local system $L_\chi$ on
$O_\chi$. Identify $\mfg^* \cong \mfg$ using the Killing form, so that
$O_\chi \subseteq \mfg$, which allows us to consider irreducible
representations $\chi$ such that $O_\chi = G \cdot e$.
\begin{corollary}\cite[Corollary 2.5]{BS-KpncSfc} The Hilbert-Poincar\'e polynomial of $\HP_0(\caO(S_e^0))$ is 
\[
y^{\dim G \cdot e}
  \sum_{\chi \in \text{Irrep}(W) \mid O_\chi=G \cdot e}
  \dim L_\chi \cdot K_{\mathfrak{g},\chi}(y^{-2}).
\]
\end{corollary}

\subsection{The hypertoric case}
Let $X$ be a
hypertoric cone as in Example \ref{ex:hypertoric},
which admits a symplectic resolution. Let $\mathcal{A}$ be the
associated hyperplane arrangement (in $\bC^{\dim X/2}$) with
$|\mathcal{A}|$ linear hyperplanes (whose normal vectors span
$\bC^{\dim X/2}$). Let
$\Phi_{\mathcal{A}}(x,y,b) :=
\Phi_{\mathcal{A}}(x,y,b_1,\ldots,b_{\mathcal{A}})$
be the polynomial defined by G.~Denham in \cite{Den-cLTc} using the combinatorial
Laplacian.
\begin{theorem}\cite[Theorem 6.1]{PS-pdrhhvnc} $h(\HP_*^{DR}(X);x,y) = y^{-\dim X}
\Phi_{\mathcal{A}}(x^2-1,y^{-2}-1,y^2)$.
\end{theorem}
The above formula is proved via the Tutte polynomial and symplectic
leaves, which are interesting in their own right. Let
$T_{\mathcal{A}}(x,y)$ denote the Tutte polynomial of the
arrangement. The symplectic leaves of $X$ are indexed by (coloop-free)
flats $F \subseteq \mathcal{A}$, cf.~e.g.,~\cite{PW-ih}. For each such
we can define the restriction, $\mathcal{A}^F$, by intersecting with
all the hyperplanes in $F$, and the localization, $\mathcal{A}_F$, by
dividing by the intersection of these hyperplanes.  Then the above
formula follows from the following one:
\[
h(\HP_*^{DR}(X);x,y) = y^{-\dim X} \sum_F T_{\mathcal{A}^F}(x^2,0)
T_{\mathcal{A}_F}(0,y^{-2}) y^{2|F|}.
\]

\subsection{Conjectural description for conical symplectic resolutions}\label{ss:conj-desc}
Finally, we sketch a conjectural description of the bigrading on
$\HP^{DR}_*(X)$ in terms of a deformation of resolution (which
enhances Conjecture \ref{con:sympres}(b) for conical symplectic
resolutions to incorporate the weights).  Let
$\rho: \widetilde{X} \to X$ be a projective conical symplectic
resolution. Let $\widetilde{\mathcal{X}} \to \mathcal{X}$ be a
$\bC^{\times}$-equivariant twistor deformation of $\rho$ over
$\bC = \Spec \bC[t]$ 
(see Remark
\ref{r:conj-flat}). Note that the degree of $t$ equals the degree of
the symplectic form on $\widetilde{X}$; call this degree $d$.

Let $\widetilde \theta: \widetilde{\mathcal{X}} \to \bC$ 
and $\theta: \mathcal{X} \to \bC$ be the projections to the base, and
let $\mathcal{X}_t := \theta^{-1}(t)$ and $\widetilde{\mathcal{X}}_t := \widetilde{\theta}^{-1}(t)$ be the fibers. By the argument of \cite[\S 4.2]{Slflsgs}, $\widetilde \theta$ is a topological fiber bundle (thanks to Y.~Namikawa for pointing this out). 
Thus the family of vector spaces $H^i(\widetilde{\mathcal{X}}_t)$
is an algebraic vector bundle on $\bC$ for each $i$ equipped with the
Gauss-Manin connection (corresponding to the right $\caD$-module
$H^{i-\dim X} \widetilde \theta_* \Omega_{\widetilde {\mathcal X}}$). Using this
connection we can uniquely trivialize the vector bundle (up to an
overall scaling) and identify it with
$H^i(\widetilde X) \otimes \bC[t]$.
(Note that the family of varieties $\widetilde{\mathcal{X}}$ is
nontrivial, since $\widetilde{\mathcal{X}}_t$ is affine for generic
$t$ but not for $t=0$.)

Let $\Omega_{\text{cl}}^i$ be the space of holomorphic
fiberwise closed differential forms on $\widetilde{\mathcal{X}}$,
i.e., cycles in the complex of holomorphic differential
forms modulo $dt$, for $t$ the function $\widetilde{\mathcal{X}} \to \bC$
above.
There is a natural map
\begin{equation}
\Phi: \Omega_{\text{cl}}^i \to 
H^i(\widetilde{X}) \otimes \bC[t].
\end{equation}
Let $K_i := \operatorname{coker}(\Phi)$.  Then, $K_i$ is
finite-dimensional and concentrated at $t=0$.  Thus, as a
$\bC[t]$-module, $K_i$ is a direct sum of Jordan blocks
$\bC[t]/t^{\phi(i,j)}$, for $j=1, 2, \ldots, \dim(H^i(\widetilde X))$
and $\phi(i,j) \in \bZ_{\geq 0}$ (where some of the Jordan blocks are
allowed to be zero, in which case we set $\phi(i,j)= 0$).  Recall that
$d$ is the degree of the generic symplectic form on $X$, which equals
the degree of $t$.
\begin{conjecture}
$h(\HP_*^{DR}(X); x,y) = y^{-d \cdot \dim X/2} \sum_{i,j} x^{i} y^{d \cdot \phi(\dim X - i,j)}$.
\end{conjecture}
The cohomological degree zero case, i.e.,
$h(\HP_0^{DR}(X);y) = y^{-d \cdot \dim X/2} \sum_j y^{d \cdot \phi(\dim X,j)}$, should follow from Conjecture \ref{con:sympres}(a) using the direct interpretation of 
$\HP_0^{DR}(X)\cong \HP_0(\caO(X))$ via
functions on $X$.   The difficulty appears to be in relating
$\HP_i^{DR}(X)$ for $i> 0$ to fiberwise closed differential forms on
the family.
\begin{remark} \label{r:deg-mult-d} Note that the conjecture predicts
  in particular that $\HP_*^{DR}(X)$ lies only in weights which are
  multiples of $d$.  This is not true in general if $X$ does not admit
  a symplectic resolution, even if it is a symplectic singularity (and
  hence has finitely many leaves): in \cite[Appendix A]{hp0bounds},
  many examples are constructed of $X=V/G$ with $V$ a symplectic
  vector space and $G < \Sp(V)$ finite, such that $\HP_0(\caO(X))$ is
  nonzero in degree three (the smallest dimension of $V$ in these
  examples is $12$). Since the Poisson bracket on $V/G$ is the one
  coming from $V$, having degree $-2$, in this case we have $d=2$ even
  though the weights are not all even.
\end{remark}
We explain how to verify the conjecture in the case of a du Val singularity
$X \cong \bC/\Gamma$ for $\Gamma < \SL(2,\bC)$ finite.  Then $\HP_*^{DR}(X)$ is only nonzero
in degrees zero and two, and $\HP_2^{DR}(X) = \bC$ occurring in weight
$-d$, by \eqref{e:m-vg} together with \cite[Theorem
5.1]{PS-pdrhhvnc} (note that $H^0(\widetilde X)$ is spanned by the constant $1$, so $\phi(0,1)=0$). Thus we only have to check the conjectural formula in degree zero. Let
$X = \{f=0\}$
where $f$ is the corresponding equation listed in Example \ref{ex:kl-ell}.
 Let the
family be $\mathcal{X} = \{f-t^{h}\} \subseteq \bC^4=\Spec \bC[x_1,x_2,x_3,t]$, where $h$ is the Coxeter number associated to the Dynkin diagram listed in the example (note that $h|t|=|f|$); this can be seen to be a twistor deformation.
Let $\omega = dx_1 \wedge dx_2 \wedge dx_3 / df$ 
be the fiberwise generic symplectic form on $\mathcal{X}$ (constant in $t$), and $\widetilde{\omega}$ be the fiberwise symplectic form on $\widetilde{\mathcal{X}}$, which generically is the pullback $\rho^*\omega$. Then we
have isomorphisms $\HP_0(X) \to H^2(X_t)$ for $t \neq 0$ given by $[g] \mapsto [g \omega]$.   
If $g$ is homogeneous, then $g \omega$ is a homogeneous form which has degree
$|g|+|\omega|=|g|+d$. Therefore the order of vanishing of
$[g \widetilde{\omega}]$ in cohomology at $t=0$ must be precisely
$|g|/d+1$.  It follows that a homogeneous basis of $\HP_0(\caO(X))$ produces forms $[g \widetilde{\omega}]$ which restrict on each fiber $t \neq 0$ to a basis of the cohomology and which vanish on cohomology at $t=0$ to order precisely $|g|/d+1$.  This is exactly what is required for the formula to hold. To complete the proof we have only to show that the elements $g \widetilde{\omega}$ span all
fiberwise (closed) two-forms on $\widetilde{\mathcal{X}}$ modulo fiberwise
exact forms.  To see this, observe that all fiberwise two-forms on
$\widetilde{\mathcal{X}}$ are of the form $g \widetilde{\omega}$ for
some
$g \in \Gamma(\widetilde{\mathcal{X}},\caO_{\widetilde{\mathcal{X}}})
= \caO(\mathcal{X})$,
and that they are fiberwise exact if and only if
$g \in \{\caO(X),\caO(X)\} \otimes_{\bC} \bC[t] \subseteq 
\caO(X) \otimes_{\bC} \bC[t] \cong \caO(\mathcal{X})$.


\appendix
\section{Background on D-modules}
In this appendix, we recall without proofs one way to define the
D-modules we need, via Kashiwara's equivalence (for a reference, see,
e.g., \cite{Hotta, Kash}; another approach, via crystals, can be found
in \cite{BD,GR-cDm}). Then, we recall the definition of holonomic
D-modules and the theorem that they are preserved by direct and
inverse image.
\subsection{Definition of D-modules on singular varieties}\label{ss:rdmod}
\begin{definition}
  Let $V$ be a smooth affine variety. Then a right D-module on
  $V$ is a right module for the ring $\caD(V)$ of differential
  operators on $V$ with polynomial coefficients.  If $V$ is not
  necessarily affine, but still smooth, then a right D-module is
  defined to be a sheaf of right modules over the sheaf $\caD_V$ of
  rings of differential operators.  A D-module is called
  quasi-coherent if the underlying $\caO_V$-module is quasi-coherent.
  Let $\text{mod}-\caD_V$ denote the category of quasi-coherent right
  D-modules on $V$.
\end{definition}
\begin{definition}
  Given a closed subset $X \subseteq V$ of a smooth variety $V$
  with ideal sheaf $\mathcal{I}_X$, and a right D-module $M$ on
  $V$, we say that $V$ is supported on $X$ if, for every open affine subset
  $U \subseteq V$ and all local sections $s \in \Gamma(U, M)$ and $f \in \Gamma(U,\mathcal{I}_X)$, there
  exists $N \geq 1$ such that
  $s \cdot f^N = 0$. 
\end{definition}
We caution that the above notion of support, which takes place on $V$,
is completely different from the notion of characteristic variety (singular support) which we
will use later, which takes place on $T^*V$.
\begin{definition}
  Suppose $X$ is an arbitrary (not necessarily smooth or affine)
  variety equipped with a closed embedding $i: X \to V$ into a smooth
  variety $V$. Then a right D-module on $X$ with respect to $i$,
  is defined to be a right D-module on $V$ which is supported on
  $i(X)$.  It is quasi-coherent if the D-module on $V$ is
  quasi-coherent.
\end{definition}
\begin{remark} In fact, not every variety admits a closed embedding
  into a smooth variety, so this definition cannot be used to define
  D-modules on arbitrary varieties.
\end{remark}
In the case that $X$ is affine,
this definition does not depend on the choice of closed embedding up to
canonical equivalence because of the following theorem of
Kashiwara. Given a closed embedding $i: Z \to V$ of smooth varieties, let
$\text{mod}_Z-\caD_V$ denote the category of right $\caD_V$-modules supported on $i(Z)$.  Then there are
functors $i_{\natural}: \text{mod}-\caD_Z \to \text{mod}_Z-\caD_V$ and
$i^{\natural}: \text{mod}_Z-\caD_V \to \text{mod}-\caD_Z$, given by, for $i_\bullet$ and $i^\bullet$ the direct and inverse image of sheaves of vector spaces,
\begin{equation}\label{e:dirfl}
i_{\natural}(M)
= i_\bullet(M \otimes_{\caD_Z} (\caO_Z \otimes_{i^\bullet\caO_V} i^\bullet\caD_V)),
\end{equation}
\begin{equation}\label{e:infl}
  i^{\natural}(M)
  := \Hom_{i^\bullet\caO_V}(\caO_Z, i^\bullet M),
\end{equation}
with canonical right D-module structures. 
\begin{theorem}\label{t:ifl-thm}
  Suppose that $i: Z \to V$ is a closed embedding of smooth varieties.
  Then the functors $i_{\natural}$ and $i^{\natural}$ above are mutually
  quasi-inverse equivalences. 
\end{theorem}
\begin{theorem}\label{t:ifl-thm2}
  Let $X$ be an arbitrary variety. 
 Suppose that $i_1: X \to V_1$ and $i_2: X \to V_2$
  are two closed embeddings with $V_1, V_2$ smooth, and that there
exists
  a third smooth variety $V_3$ together with a commuting diagram of
  closed embeddings,
\[
\xymatrix{
  & V_1 \ar[d]^{i_{13}} \\
  X \ar[ru]^{i_1} \ar[r]^{i_3} \ar[rd]^{i_2} & V_3 \\
  & V_2 \ar[u]_{i_{23}}. }
\]
Then the functors $i_{23}^{\natural} \circ (i_{13})_{\natural}$ and
$i_{13}^{\natural} \circ (i_{23})_{\natural}$ define mutually inverse
equivalences between the categories of quasi-coherent right
D-modules on $X$ with respect to $i_1$ and $i_2$. Moreover this
does not depend on the choice of $i_3$ and $V_3$, and the composition
of the equivalences from $i_1$ to $i_2$ to $i_3$ and back to $i_1$ is
the identity functor.
\end{theorem}
In the case that $X$ is an affine variety, there always exists an embedding $X \to V$ into  an affine space, and given two such embeddings $X \to V_1$ and $X \to V_2$ we
 can always find a third affine space $V_3$ such that we obtain a commuting diagram of embeddings as above.  Therefore we conclude:
 \begin{corollary} If $X$ is an affine variety, there is a category $\DmodX$ of
   quasi-coherent D-modules on $X$ which is canonically equivalent to
   the category of quasi-coherent right D-modules on $V$ supported on
   $i(X)$ for every choice of closed embedding $i: X \to V$ with $V$ an
   affine space (or a smooth affine variety).
\end{corollary}
Gluing these categories together by Theorem \ref{t:ifl-thm2}, we obtain a canonical category on general varieties:
\begin{corollary} For a general variety $X$, there is a canonical abelian
  category $\DmodX$ of quasi-coherent D-modules on $X$ such that for
  every open affine subset $U \subseteq X$, there is a canonical exact restriction functor $\DmodX \to \caD-\text{mod}_{U}$.
\end{corollary}
\begin{corollary} Given any embedding $i: X \into V$ of a variety $X$
  into a smooth variety $V$, there are canonical equivalences
$i_{\natural}: \DmodX \to \text{mod}_X-\caD_V$ and
$i^{\natural}: \text{mod}_X-\caD_V \to \DmodX$.
\end{corollary}


\begin{remark}\label{r:lrdmod}
  On a singular variety $X$ we refer only to D-modules, without
  specifying right or left, for the following reason.
  The category
$\DmodX$ of D-modules on $X$ is abstractly defined only up to canonical equivalence, and in general does not identify with modules over any sheaf of rings on $X$ itself.  Since the categories of left and right D-modules on a smooth variety $V$ are themselves canonically equivalent (via the functors $M \leftrightarrow M \otimes_{\caO_V} \Omega_V$ for $M$ a left D-module and $M \otimes_{\caO_V} \Omega_V$ the corresponding right D-module), using either one yields 
the same definition of the category $\DmodX$. The objects of this canonical
category $\DmodX$ can be thought of equivalently either as local collections of left D-modules or of right D-modules on embeddings of open subsets of $X$ into smooth varieties.  
\end{remark}
\subsection{The D-module $\caD_X$} \label{ss:dx}

  There is a global sections functor for D-modules on singular varieties which is defined as follows.
  In the case $X$ is equipped with a closed embedding $i: X \to V$
  into a smooth variety $V$, then
  $\Gamma_{\caD}(X,M) := \Hom_{i^{\bullet}\caO_V}(\caO_X,
  i^{\bullet}i_\natural M)$,
  for $M$ a D-module on $X$ and $i_\natural M$ the associated right
  D-module on $V$ supported on $i(X)$.  This is the subspace
  of the global sections of $i_\natural M$ as a sheaf on $V$,
  $\Gamma(X,M)$, which are scheme-theoretically supported on $i(X)$,
  i.e., locally annihilated by the ideal sheaf of $X$.
  This produces an $\caO(X)$-module, and by Kashiwara's equivalence it
  does not depend on the choice of embedding. 

  For a general variety, we can define the global sections functor
  on 
  D-modules by gluing the functor on affine varieties (which by
  definition embed into smooth varieties).  This is well-defined
  since, for $U_1 \subseteq U_2$ affine,
  $\Gamma_{\caD}(U_1, M|_{U_1}) = \Gamma_{\caD}(U_2, M|_{U_2})
  \otimes_{\caO(U_2)} \caO(U_1)$.
  We therefore obtain an $\caO(X)$-module, $\Gamma_{\caD}(X,M)$, for
  an arbitrary variety $X$. In the case $X$ embeds into a smooth
  variety, we recover the same answer as before (by restricting the
  embedding to affine subvarieties). We caution, however, that
  \emph{even when $X$ is affine, the global sections functor is not in
    general exact}.


  Next, given any (not-necessarily smooth) variety $X$, there is a
  canonical quasi-coherent D-module, denoted by $\caD_X$, such that
  $\Hom(\caD_X, N) = \Gamma_{\caD}(X, N)$ for all D-modules $N$ on
  $X$, i.e., $\caD_X$ represents the functor of global sections.  It
  may be defined as follows. Given any open (affine) subset
  $U \subseteq X$ and closed embedding $i: U \to V$ into a smooth
  variety, let $\mathcal{I}_U$ be the ideal sheaf of $i(U)$.  Then we
  have the D-module $\mathcal{I}_U \cdot \caD_V \setminus \caD_V$
  supported on $i(U)$, and we may set
  $\caD_U := i^\natural(\mathcal{I}_U \cdot \caD_V \setminus \caD_V)$.
  One may check explicitly that the definition does not depend on the
  choice of closed embedding and that
  $\Hom(\caD_U, N) = \Gamma_{\caD}(U, N)$ for all D-modules $N$ on
  $U$. Moreover, the definition is local:
  $\caD_{U_1} = \caD_{U_2}|_{U_1}$ for $U_1 \subseteq U_2$.
  We see therefore that these glue to a D-module, $\caD_X$, on $X$.
  (Note that for quasi-projective varieties one need not worry about
  gluing, taking $U=X$.)  Even when $X$ is affine, this D-module is
  not, in general, projective, which explains why the global sections
  functor is not, in general, exact.

  In the case that $X$ is smooth, we can use the identity embedding
  $X \to X$ and consider D-modules to be modules over the ring of
  differential operators. In this case $\caD_X$ identifies with the
  sheaf of differential operators on $X$ as a right module over
  itself. In general, though, $\caD_X$ is only given by an assignment
  to each open affine subset $U \subseteq X$ together with an
  embedding $U \to V$ into an affine space of a sheaf $\caD_U$ on $V$
  (which by Kashiwara's equivalence does not depend on the embedding);
  so we cannot compare $\caD_X$ with the sheaf of rings of
  differential operators on $X$, as the two are different types of
  objects.

If $X$ is affine, it is actually true, although nontrivial, that the
global sections $\Gamma_\caD(\caD_X) = \Hom(\caD_X,\caD_X)$ of
$\caD_X$ identify with Grothendieck's ring of differential operators
on $X$, but it is still not true that quasi-coherent D-modules
are the same as right modules over this ring. In fact, for general
singular affine varieties $X$, the category of D-modules on $X$
need not be equivalent to the category of modules over any ring.

\subsection{Holonomic D-modules}
\begin{definition}
  A coherent right D-module on a smooth affine variety $V$ is a
  finitely-generated quasi-coherent right $\caD(V)$-module.  A coherent
  D-module $M$ on an arbitrary affine variety $X$  is one such that,
for any (equivalently every) closed embedding $i: X \to V$ into a smooth affine variety, the corresponding D-module on $V$ supported on $i(X)$ is coherent.  A coherent D-module $M$ on an arbitrary variety $X$ is one such that the restriction of $M$ to every open affine subset $U \subseteq X$ is coherent.
\end{definition}
Now recall that, if $V$ is a smooth affine variety, then the ring
$\caD(V)$ of differential operators is equipped with a filtration
$\caD(V) = \bigcup_{m \geq 0} \caD_{\leq m}(V)$ by order of operator
such that $\gr(\caD(V)) := \bigoplus_{m \geq 0} \caD_{\leq m}(V) /
\caD_{\leq (m-1)}(V)$ is identified with the algebra $\caO(T^* V)$ of
functions on the total space of the cotangent bundle $T^* V$ of $V$.
\begin{definition}
  Given a quasi-coherent right D-module $M$ on a smooth affine variety $V$, a
  good (nonnegative) filtration is an filtration $M_{\leq 0} \subseteq M_{\leq 1} \subseteq \cdots$ of subsets of $M$ which is exhaustive ($M = \bigcup_{m \geq 0} M_{\leq m}$), with $M_{\leq m} \caD_{\leq n} \subseteq M_{\leq n}$ for $m,n \geq 0$, and such that $\gr M = \bigoplus_{m \geq 0} M_{\leq m} / M_{\leq m-1}$ is a finitely-generated $\gr \caD(V) = \caO(T^*V)$-module (with $M_{\leq -1} := 0$).
\end{definition}

As explained in, e.g., \cite[Theorem 2.13]{Hotta}, every coherent
right D-module $M$ admits a good filtration:
let $M_{\leq 0}$ be any finite-dimensional subspace
which generates $M$, and set $M_{\leq m} := \caD(V)_{\leq m} M_{\leq 0}$ for all $m
> 0$.  Conversely, all quasi-coherent D-modules admitting good
filtrations are coherent.
\begin{definition}\label{d:ss}
  Given a coherent right D-module $M$ on a smooth affine variety $V$
  equipped with a good filtration, the characteristic variety, $\SiSu(M)$,
  is defined to be the set-theoretic support of $\gr M$ over $T^* V$.
\end{definition}
One can check that the characteristic variety does not depend on the choice
of good filtration (this follows from \cite[Theorem 2.13]{Hotta}).
This allows one to extend the notion of characteristic variety to the
not-necessarily affine case, since the characteristic variety is defined
locally. We conclude that every coherent right D-module on a
smooth variety has a well-defined characteristic variety.
\begin{definition}
  A (nonzero) holonomic right D-module on a smooth irreducible variety
  $V$ is a coherent right D-module whose characteristic variety has
  dimension equal to the dimension of $V$.
\end{definition}
In fact, the characteristic variety is well known to be a coisotropic
subvariety of $T^* V$ (by a theorem of Sato, Kawai, and Kashiwara \cite{SKK-mfpde}; see also \cite{Gab-icv}),
so the (nonzero) D-module is
holonomic if and only if this is Lagrangian (and hence has the minimal
possible dimension). By convention, the zero module is also holonomic.

\subsection{Direct and inverse image}\label{ss:d-i-im}
Given a map $f: X \to Y$ of smooth varieties, we have natural functors
$f^!: D^b(\text{mod}-\caD_Y) \to D^b(\text{mod}-\caD_X)$ and
$f_*: D^b(\text{mod}-\caD_X) \to D^b(\text{mod}-\caD_Y)$. These are
not, in general, the derived functors of any functors on abelian
categories. However, when $f$ is affine, $f_*$ is the derived functor
of a right exact functor, and when $f$ is an open embedding, $f_*$ is
the derived functor of a left exact functor (and in this case $f^!$ is
the exact restriction functor).  In particular, if $f$ is a closed
embedding, then $f_*$ is the derived functor of the exact functor
$f_\natural$ (now viewed as having target equal to all quasi-coherent
D-modules on $Y$).  Also, when $f$ is a closed embedding, $f^!$ is the
derived functor of a left exact functor (which is given by the same
definition as $f^\natural$, now defined on all $\caD$-modules on $Y$
rather than merely those set-theoretically supported on $f(X)$).
In particular, when $f$ is a closed embedding, $f_*$ coincides with
$f_\natural$ (more precisely, $f_\natural M$ is the cohomology of
$f_* M$), and $f^!$ coincides with $f^\natural$ on $\caD$-modules
supported on $f(X)$ (more precisely, $f^\natural M$ is the cohomology
of $f^! M$ if $M$ is supported on $f(X)$).  
The definitions are:
\begin{equation}
\label{e:dirdir}
f_* M := Rf_\bullet(M \otimes^L_{\caD_X} \caD_{X \to Y}), \quad  \caD_{X \to Y} := \caO_X \otimes_{f^\bullet \caO_Y}
f^{\bullet} \caD_Y;
\end{equation}
\begin{equation}\label{e:imdir}
  f^!(M):=f^\bullet(M) \otimes^L_{f^\bullet \caD_Y} \caD_{Y
    \leftarrow X}[\dim X - \dim Y], \quad
  \caD_{Y \leftarrow X} := \Omega_X \otimes_{\caO_X} \caD_{X \to Y}
  \otimes_{f^\bullet\caO_Y} f^\bullet \Omega_Y^{-1}.
\end{equation}
See \cite[\S 1.3, 1.5]{Hotta}.
\begin{theorem}(e.g., \cite[Theorem 3.2.3]{Hotta}) \label{t:holonomic-push}
Let $f: X \to Y$ be
  a map of smooth varieties and $M$ and $N$ bounded complexes of
  quasi-coherent right D-modules on $X$ and $Y$ whose cohomology
  D-modules are holonomic.  Then $f^! N$ and $f_* M$ are
  bounded complexes whose cohomology D-modules are holonomic.
\end{theorem}
\begin{corollary}
  Let $X$ be an arbitrary variety and $M$ a quasi-coherent 
  D-module on $X$.  Then, given two closed embeddings
  $i_1: X \to V_1$ and $i_2: X \to V_2$, $(i_1)_{\natural} M$ is
  holonomic if and only if $(i_2)_{\natural} M$ is.
\end{corollary}
Therefore, we can make the following definition:
\begin{definition}
  A quasi-coherent D-module $M$ on an affine variety $X$ is called
  holonomic if, for any closed embedding $i: X \to V$ into a smooth
  affine variety, $i_\natural M$ is holonomic. A D-module $M$ on an
  arbitrary variety $X$ is called holonomic if, for every open affine
  subset $U \subseteq V$, the restriction $M|_U$ of $M$ to $U$ is
  holonomic.
\end{definition}
With this definition in place, Theorem \ref{t:holonomic-push}
immediately generalizes to arbitrary varieties.  Namely, if $X \to Y$
is an arbitrary map of varieties, then one obtains canonical functors
$f^!: \DbmodY \to \DbmodX$ and $f_*: \DbmodX \to \DbmodY$ preserving
holonomicity:
\begin{corollary}
  Let $f: X \to Y$ be an arbitrary map of varieties and $M$ and $N$
  bounded complexes of quasi-coherent  D-modules on $X$ and
  $Y$ whose cohomology D-modules are holonomic.  Then $f^! N$ and
  $f_* M$ are bounded complexes whose cohomology D-modules are
  holonomic.
\end{corollary}
Observe that, when $X$ is a point, a holonomic D-module is merely
a finite-dimensional vector space (since finite generation reduces to
finite-dimensionality, and the support condition is trivial). We
therefore deduce:
\begin{corollary}\label{c:fd-pfwd}
  If $M$ is a complex of quasi-coherent D-modules on a
  variety $X$ with holonomic cohomology and $\pi: X \to \pt$ is the
  projection, then $\pi_* M$ is a complex with finite-dimensional
  cohomology. In particular, if $M$ is a holonomic D-module, then
  $H^0\pi_{*} M$ is a finite-dimensional vector space.
\end{corollary}

\bibliographystyle{amsalpha}
\bibliography{../../research/bibtex/master}

\end{document}